\documentclass[11pt]{article} \usepackage[utf8]{inputenc}
\usepackage[T1]{fontenc}
\usepackage{lmodern}
\usepackage[english]{babel}

\usepackage[a4paper,vmargin={3.5cm,3.5cm},hmargin={2.5cm,2.5cm}]{geometry}
\usepackage{mathpazo}
\usepackage[font={small,sf}, labelfont={sf,bf}, margin=1cm]{caption}
\usepackage[pdftex,colorlinks=true]{hyperref}
\usepackage[expansion]{microtype}
\usepackage[pdftex]{color,graphicx}
\usepackage{amsmath,amsfonts,amssymb,amsthm,mathrsfs}
\usepackage{stackrel}

\theoremstyle{plain}
\newtheorem{theorem}{Theorem}
\newtheorem{corollary}{Corollary}
\newtheorem{proposition}{Proposition}
\newtheorem{lemma}{Lemma}
\theoremstyle{definition}
\newtheorem*{definition}{Definition}
\newtheorem*{remark}{Remark}
\newtheorem*{ack}{Acknowledgments}

\hypersetup{
    pdfauthor   = {Laurent Menard},%
    pdftitle    = {Volumes in the UIPT}}

\begin{document}

\title{\bf Volumes in the Uniform Infinite Planar Triangulation:\\
from skeletons to generating functions}
\author{\textsc{Laurent Ménard}\footnote{laurent.menard@normalesup.org}\\Modal'X, Universit\'e Paris Ouest
and LiX, \'Ecole Polytechnique}
\date{}
\maketitle


\begin{abstract}
We develop a method to compute the generating function of the number of vertices inside certain regions of the Uniform Infinite Planar Triangulation (UIPT). The computations are mostly combinatorial in flavor and the main tool is the decomposition of the UIPT into layers, called the skeleton decomposition, introduced by Krikun \cite{Kr}. In particular, we get explicit formulas for the generating functions of the number of vertices inside hulls (or completed metric balls) centered around the root, and the number of vertices inside geodesic slices of these hulls. We also recover known results about the scaling limit of the volume of hulls previously obtained by Curien and Le Gall by studying the peeling process of the UIPT in \cite{CLG}.
\end{abstract}





\section{Introduction and main results}

The probabilistic study of large random planar maps takes its roots in theoretical physics, where planar maps are considered as approximations of universal two dimensional random geometries in Liouville quantum gravity theory (see for instance the book \cite{ADJ}). In the past decade, a lot of work has been devoted to make rigorous sense of this idea with the construction and study of the so-called Brownian map. The surveys \cite{LG,Mi} will give the interested reader a nice overview of the field as well as an up to date list of references.

Since they are instrumental in every proof of convergence to the Brownian map, the most successful tools to study random planar maps are undoubtedly the various bijections between certain classes of maps and decorated trees. The search for such bijections was initiated by Cori and Vauquelin \cite{CV} and perfected by Schaeffer \cite{Sch}. Since then, a lot of bijections in the same spirit have been discovered, (see in particular the one by Bouttier, Di Francesco and Guitter \cite{BdFG}). These bijections are particularly well suited to study metric properties of large random maps (see the seminal work of Chassaing and Schaeffer \cite{CS}), and they have lead to the remarkable proofs of convergence in the Gromov-Hausdorff topology of wide families of random maps to the Brownian map by Le Gall \cite{LGbm} and Miermont \cite{Mibm} independently, paving the way to other results of convergence \cite{Ab,AA,BJM}.

Another very powerful tool to study random maps is the so called \emph{peeling process} -- informally a Markovian exploration procedure -- introduced by Watabiki \cite{W} and used immediately by Watabiki and Ambj{\o}rn to derive heuristics for the Hausdorff dimension of random maps in \cite{AW}. Probabilists started to show interest in this procedure a bit later, starting by Angel \cite{A}, who formalized it in the setting of the Uniform Infinite Planar Triangulation (UIPT). Since then, this process has received growing attention and proved valuable to study not only the geometry of random maps \cite{A,BCK,B,CLG}, but also random walks \cite{BC}, percolation \cite{A,AC,MN,R}, and even, to some extent, conformal aspects \cite{C}.

\bigskip

In this work we will use another tool, introduced by Krikun \cite{Kr}, to study the UIPT, called the skeleton decomposition. Before we present this tool, let us recall that a planar map is a proper embedding of a connected multi-graph in the two dimensional sphere, considered up to orientation preserving homeomorphisms. The maps we consider will always be rooted (they have a distinguished oriented edge), and we will focus on rooted triangulations of type I in the terminology of Angel and Schramm \cite{AS}, meaning that loops and multiple edges are allowed and that every face of the map is a triangle. The UIPT is the infinite random lattice defined as the local limit of uniformly distributed rooted planar triangulations with $n$ faces as $n \to \infty$ (see Angel and Schramm \cite{AS}). We will denote the UIPT by $\mathbf T _{\infty}$ and, if $M$ is a (finite) planar map, we will denote its number of vertices by $|M|$.

For every integer $r \geq 1$, the ball $B_r (\mathbf{T}_{\infty})$ is the submap of $\mathbf{T}_{\infty}$ composed of all its faces having at least one vertex at distance stricly less that $r$ from the origin of the root edge. Since the UIPT is almost surely one ended, of all the connected component of $\mathbf{T}_{\infty} \setminus B_r(\mathbf{T}_{\infty})$, only one is infinite and the hull $B_r^\bullet(\mathbf{T}_{\infty})$ is the complement in $\mathbf{T}_{\infty}$ of this unique infinite connected component (see Figure \ref{fig:hull} for an illustration). The layers of the UIPT are the sets $B_r^\bullet(\mathbf{T}_{\infty}) \setminus B_{r-1}^\bullet(\mathbf{T}_{\infty})$ for $r \geq 1$. The skeleton decomposition  of the UIPT roughly states that the geometry of the layers of the UIPT is in one-to-one correspondance with a critical branching process and a collection of independent Boltzmann (or free) triangulations with a boundary (see Figure \ref{fig:genealogy}). We will give a detailed presentation of this decomposition in Section \ref{sec:prel}.

\begin{figure}[ht!]
\begin{center}
\includegraphics[width=0.8\textwidth]{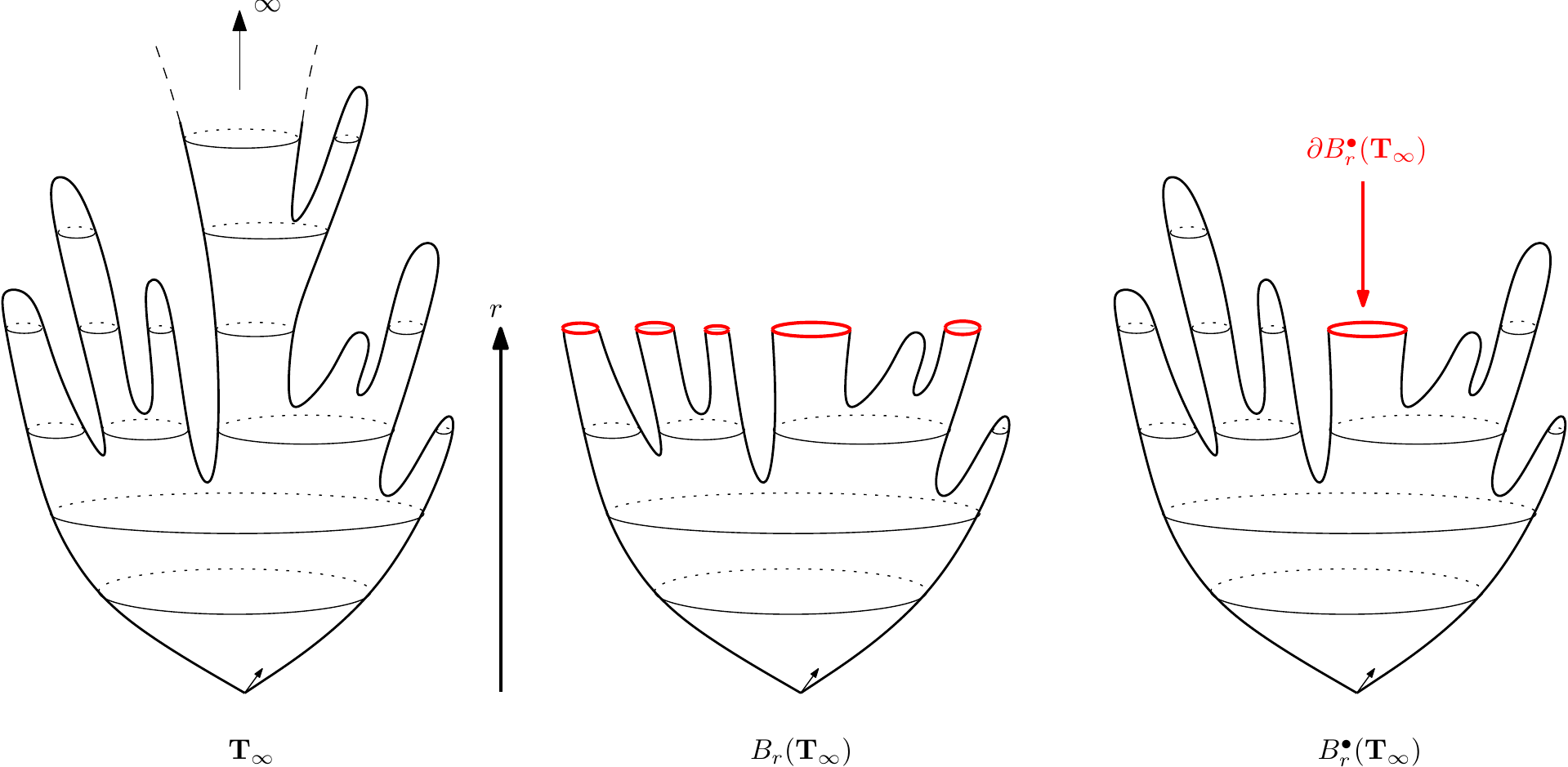}
\caption{\label{fig:hull} Illustration of the ball of radius $r$ in the UIPT and the corresponding hull.}
\end{center}
\end{figure}

This decomposition was used by Krikun in \cite{Kr} to study the length of the boundary of the hulls $B_r^\bullet(\mathbf{T}_{\infty})$ of the UIPT and in $\cite{K}$ for similar considerations on the Uniform Infinite Planar Quadrangulation. Since then, this decomposition has not received much attention with the notable exception of the recent work by Curien and Le Gall \cite{CLGfpp}, where it is used to study local modifications of the graph distance in the UIPT.

\bigskip

We will use the skeleton decomposition of the UIPT to get exact expressions for the generating functions of the number of vertices inside certain regions of hulls, starting with the hulls themselves.

\begin{theorem}
\label{th:volhull}
For any $s \in [0,1]$ and any nonnegative integer $r$ one has
\[
\mathbb{E} \left[ s^{|B^{\bullet}_r(\mathbf{T}_{\infty})| } \right] =
3^{3/2}
\frac{
\cosh 
\left( 
\sinh^{-1} 
\left( \sqrt{3\frac{(1-t)}{t}} \right)
+ r \cosh^{-1} 
\left( \sqrt{\frac{3-2t}{t}} \right)
\right)}
{\left(
\cosh^2 
\left( 
\sinh^{-1} 
\left( \sqrt{3\frac{(1-t)}{t}} \right)
+ r \cosh^{-1} 
\left( \sqrt{\frac{3-2t}{t}} \right)
\right)
+2 \right)^{3/2}}
\]
where $t$ is the unique solution in $[0,1]$ of the equation $s^2 = t^2 (3-2t)$.
\end{theorem}

An easy consequence of this Theorem is the scaling limit
\[
\lim_{R \to \infty}
\mathbb{E} \left[ e^{- \lambda|B^{\bullet}_{\lfloor x R \rfloor}(\mathbf{T}_{\infty})|/R^4 } \right]
\]
already obtained in \cite{CLG} via the peeling process. Indeed, put $s=e^{- \lambda / R^4}$ and $r = \lfloor x R \rfloor$ for some $\lambda,x > 0$ and some integer $R$. Then
\[
t = 1 - \frac{\sqrt{2 \lambda /3}}{R^2} + o(R^{-2})
\]
and
\[
\cosh^{-1} 
\left( \sqrt{\frac{3-2t}{t}} \right) \sim \sqrt{\frac{3-2t}{t} - 1} \sim \frac{(6 \lambda)^{1/4}}{R}
\]
giving
\[
\lim_{R \to \infty}
\mathbb{E} \left[ e^{- \lambda|B^{\bullet}_{\lfloor x R \rfloor}(\mathbf{T}_{\infty})|/R^4 } \right] = 3^{3/2} \frac{\cosh \left( (6 \lambda)^{1/4} x \right)}{\left(\cosh^2 \left((6 \lambda)^{1/4} x \right) +2 \right)^{3/2}}
\]
in accordance with \cite{CLG,CLGplane} for type I triangulations.

\bigskip

We also get an explicit expression for the generating function of the volume of hulls conditionally on their perimeter (see Proposition \ref{prop:layervol} for a precise statement). This allows to recover the following scaling limit, already appearing in \cite{CLGplane}, Theorem 1.4, as the Laplace transform of the volume of hulls of the Brownian plane conditionally on the perimeter.
\begin{corollary}
\label{cor:hullcond}
Fix $x, \ell >0$, then, for any $\lambda >0$, one has
\begin{align*}
\lim_{R \to \infty}
\mathbb{E} & \left[ e^{- \lambda|B^{\bullet}_{\lfloor x R \rfloor}(\mathbf{T}_{\infty})|/R^4 } \middle| |\partial B^{\bullet}_{\lfloor x R \rfloor}(\mathbf{T}_{\infty})| = \lfloor \ell R^2 \rfloor \right] \\
& \quad \quad = 
x^3 \left( 6 \lambda \right)^{3/4} \frac{\cosh \left( \left( 6 \lambda \right)^{1/4} x \right)}{\sinh^3 \left( \left( 6 \lambda \right)^{1/4} x \right)}
\exp \left( - \ell \left( (6 \lambda)^{1/4} \left( \coth^2 \left( (6 \lambda)^{1/4} x \right) -\frac{2}{3}\right) - \frac{1}{x^2} \right) \right).
\end{align*}
\end{corollary}

\bigskip

Our approach also allows us to compute the exact generating function of the difference of volume between hulls of the UIPT (see Proposition \ref{prop:layervol}), and then recover one of the main results of \cite{CLG}, namely the scaling limit of the volumes of Hulls to a stochastic process. This convergence holds jointly with the scaling limit of the perimeter of the hulls and we need to introduce some notation taken from \cite{CLG} to state it.

Let $(X_t)_{t \geq 0}$ be the Feller Markov process with values in $\mathbb R _+$ whose semigroup is characterized by
\[
\mathbb E \left[ e^{- \lambda X_t} \middle| X_0 = x \right]
= \exp \left( -x \left( \lambda^{-1/2} +t/2 \right)^{-2} \right)
\]
for every $x,t \geq 0$ and $\lambda >0$. The process $X$ is a continuous time branching process with branching mechanism given by $u \mapsto u^{3/2}$. As explain in \cite{CLGplane}, one can construct a stochastic process $(\mathcal L_t)_{t \geq 0}$ with \emph{c\`adl\`ag} paths such that the time-reversed process $(\mathcal L_{(-t)-} )_{t\leq 0}$ is distributed as $X$ "started from $+\infty$ at time $-\infty$" and conditioned to hit $0$ at time $0$. We also let $(\xi_i)_{i\geq 1}$ be a sequence of independent real valued random variables with density
\[
\frac{1}{\sqrt{2 \pi x^5}} e^{-\frac{1}{2x}} \mathbf{1}_{\{x >0\}}
\]
and assume that this sequence is independent of the process $\mathcal L$. Finally we set
\[
\mathcal M_t = \sum_{s_i \leq t} \xi_i \left( \Delta \mathcal L _{s_i} \right)^2,
\]
where $(s_i)_{i \geq 1}$ is a measurable enumeration of the jumps of $\mathcal L$. We recover the following result, first proved in \cite{CLG} by studying the peeling process of the UIPT:

\begin{theorem}[\cite{CLG}, scaling limit of the hull process]
\label{th:hullproc}
We have the following convergence in distribution in the sense of Skorokhod:
\[
\left(R^{-2} |\partial B^{\bullet}_{\lfloor x R \rfloor}(\mathbf{T}_{\infty})|, R^{-4}|B^{\bullet}_{\lfloor x R \rfloor}(\mathbf{T}_{\infty})| \right)_{x \geq 0}
\xrightarrow[R \to \infty]{(d)}
\left(3^2 \cdot \mathcal L_x,4 \cdot 3^3 \cdot \mathcal M_x\right)_{x \geq 0}.
\]
\end{theorem}

As for Theorem \ref{th:volhull}, our proof is based on the skeleton decomposition of random triangulations and explicit computations of generating functions. The convergence of perimeters towards the process $\mathcal L$ was already established by Krikun \cite{Kr} using this decomposition and we prove the joint convergence of the second component.

\bigskip

Finally, we study the volume of geodesic slices of the UIPT, defined by analogy with geodesic slices of the Brownian map (see Miller and Sheffield \cite{MS}). Fix $r > 0$, and orient $\partial B_r^{\bullet} (\mathbf T_\infty)$ in such a way that the root edge of $\mathbf T_\infty$ lies on its right hand side. Now pick two vertices $v,v' \in \partial B_r^{\bullet} (\mathbf T_\infty)$, the geodesic slice $\mathbf S(r,v,v')$ is the submap of $B_r^{\bullet} (\mathbf T_\infty)$ bounded by the two leftmost geodesics (see Section \ref{sec:slices} for a precise definition) started respectively at $v$ and $v'$ to the root, and by the oriented arc from $v$ to $v'$ along $\partial B_r^{\bullet} (T_\infty)$ (See Figure \ref{fig:slice} for an illustration). Notice that $B_r^{\bullet} (\mathbf T_\infty) = \mathbf S(r,v,v') \cup \mathbf S(r,v',v)$. We will also denote by $v \wedge v'$ the vertex where the two leftmost geodesics started at $v$ and $v'$ coalesce.

For technical reasons, it will be easier to study the volume of geodesic slices minus the number of vertices on one of the two geodesics bounding it (for $S(r,v,v')$, we are talking about excluding a number of vertices between $2$ and $r + 1$). It is still possible to study the full volume of slices, but the formulas we provide will be much simpler and the number of vertices excluded is insignificant for large $r$ anyway.
\begin{theorem}
\label{th:slices}
Fix $n,r,q$ and $q_1, \ldots , q_n$ some non negative integers such that $q_1 + \cdots + q_n = q$. Conditionally on the event $\{|\partial B_r^{\bullet} (\mathbf T_\infty)| = q \}$, let $v_1$ be a vertex of $\partial B_r^{\bullet} (\mathbf T_\infty)$ chosen uniformly at random and let $v_2, \ldots , v_n$ be placed in that order on the oriented cycle $\partial B_r^{\bullet} (\mathbf T_\infty)$ such that the oriented arc from $v_j$ to $v_{j+1}$ along $\partial B_r^{\bullet} (\mathbf T_\infty)$ has length $q_j$ for every $j$ (we set $v_{n+1} = v_1$). Then, for $s_1, \ldots , s_n \in [0,1]$, one has
\begin{align*}
\mathbb{E} & \left[
\prod_{j=1}^n s_j^{|\mathbf S(r,v_j,v_{j+1})| - d(v_j,v_j \wedge v_{j+1}) -1}
\middle|  |\partial B_r^\bullet (\mathbf T_\infty) | = q
\right]\\
& \qquad \qquad \qquad =
\left(
\prod_{j=1}^n \left(t_j \, \frac{\varphi_{t_j}^{\{r\}}(0)}{\varphi^{\{r\}}(0)} \right)^{q_j}
\right) \times
\sum_{k=1}^n \frac{q_k}{q} \frac{1}{t_k} \, 
\frac{\varphi_{t_k}^{\{r\}'}(0)}{\varphi^{\{r\}'}(0)}
\frac{\varphi^{\{r\}}(0)}{\varphi_{t_k}^{\{r\}}(0)}
\end{align*}
where, for every $j \in \{ 1, \ldots ,n\}$, $t_j$ is the unique solution in $[0,1]$ of the equation $s_j^2 = t_j^2 (3-2t_j)$ and the functions $\varphi_t^{\{r\}}$ and $\varphi^{\{r\}}$ are computed explicitly in Lemma \ref{lem:ExpPhin}.
\end{theorem}

Equivalently, Theorem \ref{th:slices} states that, for each $k$, the root vertex of $\mathbf T_{\infty}$ belongs to the slice $\mathbf S(r,v_k,v_{k+1})$ with probability $q_k/q$ and that its volume has generating function
\[
\left(t_k \, \frac{\varphi_{t_k}^{\{r\}}(0)}{\varphi^{\{r\}}(0)} \right)^{q_k-1} \cdot \frac{\varphi_{t_k}^{\{r\}'}(0)}{\varphi^{\{r\}'}(0)},
\]
and that conditionally on this event, the volumes of the other slices are independent and have generating functions given by
\[
\left(t_j \, \frac{\varphi_{t_j}^{\{r\}}(0)}{\varphi^{\{r\}}(0)} \right)^{q_j}
\]
for every $j \neq k$. It is also worth noticing that the generating function of the volume of the slice containing the root vertex is exactly the same as the hull of $\mathbf T_{\infty}$ conditionally on the event $\{|\partial B_r^{\bullet} (\mathbf T_\infty)| = q_k \}$, suggesting that this slice has the same law as a hull once the two geodesic boundaries are glued.

\bigskip

Since geodesic slices do not form a growing family as the radius of the hulls grows, it is less natural to look for a scaling limit of their volume as a stochastic processes as in Theorem \ref{th:hullproc}. However, it is still quite straightforward to derive asymptotics from Theorem \ref{th:slices} and obtain
\begin{corollary}
\label{cor:slicelim}
Fix $n > 0$ an integer and $\ell,x >0$ some real numbers. Fix also $\ell_1, \ldots , \ell_n$ some non negative reals such that $\ell_1 + \cdots + \ell_n = \ell$. For every integer $R > 0$, conditionally on the event $\{|\partial B_{\lfloor x R \rfloor}^{\bullet} (\mathbf T_\infty)| = \lfloor \ell R^2 \rfloor \}$, let $v_1$ be a vertex of $\partial B_{\lfloor x R \rfloor}^{\bullet} (\mathbf T_\infty)$ chosen uniformly at random and let $v_2, \ldots , v_n$ be placed in that order on the oriented cycle $\partial B_{\lfloor x R \rfloor}^{\bullet} (\mathbf T_\infty)$ such that the oriented arc from $v_j$ to $v_{j+1}$ along $\partial B_{\lfloor x R \rfloor}^{\bullet} (\mathbf T_\infty)$ has length $\sim \ell_j R^2$ as $R \to \infty$. Then, for $\lambda_1, \ldots , \lambda_n > 0$, one has
\begin{align*}
\lim_{R\to \infty}
\mathbb{E} & \left[
\prod_{j=1}^n e^{-\lambda_j |\mathbf S(r,v_j,v_{j+1})| / R^{4}}
\middle|  |\partial B_{\lfloor x R \rfloor}^\bullet (\mathbf T_\infty) | = \lfloor \ell R^2 \rfloor
\right] \\
& \quad \quad
=
\left( \sum_{i=1}^n \frac{\ell_i}{\ell} x^3 (6 \lambda_i)^{3/4} \frac{\cosh \left((6 \lambda_i)^{1/4} x \right)}{\sinh^3 \left((6 \lambda_i)^{1/4} x \right)} \right)\\
& \quad \quad \quad \quad  \times 
\exp \left( -\sum_{i=1}^n  
 \ell_i \left( (6 \lambda)^{1/4} \left( \coth^2 \left( (6 \lambda)^{1/4} x \right) -\frac{2}{3}\right) - \frac{1}{x^2} \right) 
\right).
\end{align*}
\end{corollary}
As for Corolloray \ref{cor:hullcond}, this can be interpreted in terms of the Brownian plane: each slice has probability $\ell_i/\ell$ to contain the root, in which case its volume has the same law as the volume of the hull of the Brownian plane condionally on the perimeter being $\ell_i$. In addition, conditionally on this event, the volume of the other slices are independent and their Laplace transform is given by
\[
\exp \left( -  
 \ell_j \left( (6 \lambda)^{1/4} \left( \coth^2 \left( (6 \lambda)^{1/4} x \right) -\frac{2}{3}\right) - \frac{1}{x^2} \right) 
\right)
\]
for every $j \neq i$.

\bigskip

The paper is organized as follows. In Section \ref{sec:prel} we recall some results about the generating functions of triangulations counted by boundary length and inner vertices and we describe the decomposition of the UIPT into layers. In Section \ref{sec:hullvol} we present our method and use it to prove Theorem \ref{th:volhull} and Corollary \ref{cor:hullcond}. Section \ref{sec:hullproc} studies the difference of volume between hulls and contains the proof of Theorem \ref{th:hullproc}. Finally, Section \ref{sec:slices} studies geodesic slices and contains the proofs of Theorem \ref{th:slices} and Corollary \ref{cor:slicelim}.

\bigskip

\begin{ack}
The author would like to thank Julien Bureaux for pointing out the link with  hyperbolic functions in Lemma \ref{lem:ExpPhin}, yielding a simpler proof and a much nicer formula. The author acknowledges support form the ANR grant "GRaphes et Arbres AL\'eatoires" (ANR-14-CE25-0014) and from CNRS.
\end{ack}

\section{Preliminaries}
\label{sec:prel}

\subsection{Generating Series}

As already mentioned in the introduction, the triangulations we consider in this work are type I triangulations in the terminology of Angel and Schramm \cite{AS} -- loops and multiple edges are allowed -- and will always be rooted even when not mentioned explicitely. More precisely, we deal with triangulations with simple boundary, that is rooted planar maps (the root of a map is a distinguished oriented edge and the root vertex of a rooted map is the origin of its root edge) such that every face is a triangle except for the face incident to the right hand side of the root edge which can be any simple polygon. If the length of the boundary face is $p$, we will speak of triangulations of the $p$-gon.

One of the advantages of dealing with type I triangulations for our purpose is that triangulations of the sphere can be thought of as triangulations of the $1$-gon as already mentioned in \cite{CLGfpp}. To see that, split the root edge of any triangulation into a double edge and then add a loop inside the region bounded by the new double edge and re-root the triangulation at this loop oriented clockwise (so that the interior of the loop lies on its right hand side). Note that this construction also works if the root is itself a loop. This tranformation is a bijection between triangulations of the sphere and triangulations of the $1$-gon.

The enumeration of triangulations of the $p-$gon is now well known and can be found for example in \cite{CLGfpp,Kr}. Let $\mathcal{T}_{n,p}$ be the set of triangulations of the $p-$gon with $n$ inner vertices (\emph{i.e.} vertices that do not belong to the boundary face) and define the bivariate generating series
\begin{equation}
\label{eq:Tseries}
T(x,y) = \sum_{p \geq 1} \sum_{n \geq 0} \left| \mathcal{T}_{n,p} \right| x^n y^{p-1}.
\end{equation}
Tutte's equation reads, for $y > 0$,
\begin{equation}
\label{eq:Tutte}
T(x,y) = y + x \cdot \frac{T(x,y) - T(x,0)}{y} + T(x,y)^2.
\end{equation}
This equation can be solved using the quadratic method and the solution is explicit in terms of the unique solution of the equation
\begin{equation}
\label{eq:h}
x^2 = h(x)^2 (1-8h(x))
\end{equation}
such that $h(0) = 0$. This function $h$ seen a Taylor series has non negative coefficients and its radius of convergence is
\begin{equation}
\label{eq:rho}
\rho := \frac{1}{12 \sqrt{3}}.
\end{equation}
In addition, it is finite at its radius of convergence and
\begin{equation}
\label{eq:alpha}
\alpha := h(\rho) = \frac{1}{12}.
\end{equation}
The solution of equation \eqref{eq:Tutte} is then well defined on  $[0,\rho] \times [0,\alpha]$ and given by
\begin{equation}
\label{eq:expT0h}
T(x,0) = \frac{6h^2 + x - h}{2x}
\end{equation}
for $x \in [0,\rho]$ and
\begin{align}
\label{eq:expTxy}
T(x,y) & = \frac{y-x}{2y} + \frac{\sqrt{(y-x)^2 - 4 y^3 + 4 x y T(x,0)}}{2y} \notag\\
& = \frac{y-x}{2y} + \frac{\sqrt{x^2+y^2-4y^3+12yh^2-2yh}}{2y}
\end{align}
for $(x,y) \in [0,\rho] \times [0,\alpha]$. These expressions are compatible when taking the limit $x \to 0$ and/or $y \to 0$. Notice also that $T(\rho,\alpha)$ is finite:
\[
T(\rho,\alpha) = \frac{3-\sqrt{3}}{6}.
\]

The formulas \eqref{eq:expT0h} and \eqref{eq:expTxy} allow to compute explicitely the number of triangulations of the $p$-gon with a given number of inner vertices. However we will not need the exact formulas, only the following asymptotic expression:
\[
\left| \mathcal{T}_{n,p} \right| \underset{n \to \infty}{\sim} C(p) \rho^{-n} n ^{-5/2}
\]
for every $p \geq 1$ with
\[
C(p) = \frac{3^{p-2} p (2p)!}{4 \sqrt{2 \pi} (p!)^2}.
\]
In the following we will always denote by $T_p(x) = [y^{p-1}]T(x,y)$ the generating series of triangulations with boundary length $p$ counted by inner vertices.

\subsection{Skeleton decomposition of finite triangulations}
\label{sec:skel}

We present here the skeleton decomposition of triangulations as first defined by Krikun \cite{Kr} for type II triangulations and later by Curien and Le Gall \cite{CLGfpp} for type I triangulations. First, we need to define balls and hulls for finite triangulations.

Let $T$ be a triangulation of the sphere seen as a triangulation of the $1$-gon. For every integer $r > 0$, the ball $B_r(T)$ of radius $r$ centered at the root vertex of $T$ is the planar map obtained by taking the union of the faces of $T$ that have at least one vertex at distance less than or equal to $r-1$ from the root vertex of $T$. Now let $v$ be a distinguished vertex of $T$ and fix $r>0$ such that the distance between $v$ and the root vertex of $T$ is strictly larger than $r$. In that case, the vertex $v$ belongs to the complement of the ball $B_r(T)$ and we define the $r-$hull $B^{\bullet} _r(T,v)$ of the pointed map $(T,v)$ as the union of $B_r(T)$ and all the connected components of the complement in $T$ of $B_r(T)$ except the one that contains $v$. 

Define the boundary $\partial B^{\bullet}_r(T,v)$ of $B^{\bullet}_r(T,v)$ as the set of vertices of $B^{\bullet}_r(T,v)$ having at least one neighbour in the complement of $B^{\bullet}_r(T,v)$, with the edges joining any pair of such vertices. An important observation is that $\partial B^{\bullet}_r(T,v)$ is a simple cycle of $T$ and that its vertices are all at distance exactly $r$ from the root vertex of $T$.
The planar map $B^{\bullet}_r(T,v)$ is therefore almost a triangulation with a simple boundary, the difference being that it is rooted at the orginal root edge of $T$ instead of an edge of the boundary face. It is a special case of a triangulation of the cylinder defined in \cite{CLGfpp}:
\begin{definition}
Let $r \geq 1$ be an integer. A \emph{triangulation of the cylinder of height $r$} is a rooted planar map such that all faces are triangles except for two distinguished faces verifying:
\begin{enumerate}
\item The boundaries of the two distiguished faces form two disjoint simple cycles.
\item The boundary of one of the two distinguished faces contains the root edge, and this face is on the right hand side of the root edge. We call this face the root face and the other distiguished face the exterior face.
\item Every vertex of the exterior face is at graph distance exactly $r$ from the boundary of the root face, and edges of the boundary of the exterior face also belong to a triangle whose third vertex is at distance $r-1$ from the root face.
\end{enumerate}
For every intergers $r,p,q \geq 1$, a \emph{triangulation of the $(r,p,q)$-cylinder} is a triangulation of the cylinder of height $r$ such that its root face has degree $p$ and its exterior face has degree $q$.
\end{definition}
With that terminology, the planar maps $\Delta$ such that $\Delta = B^{\bullet}_r(T,v)$ for some integer $r$ and some pointed triangulation of the sphere $(T,v)$ are the triangulations of the $(r,1,q)-$cylinder for some integer $q \geq 1$. Triangulations of the cylinder will also allow us to describe the geometry of triangulations between hulls. More precisely, if $(T,v)$ is a pointed triangulation of the sphere and $r_2 > r_1 >0$ are two integers such that $v$ is at distance strictly larger than $r_2$ from the root vertex of $T$, we define the layer between heights $r_1$ and $r_2$ of $(T,v)$ by
\[
L_{r_1,r_2}^\bullet (T,v) = \left(B^{\bullet}_{r_2}(T,v)\setminus B^{\bullet}_{r_1}(T,v) \right) \cup \partial B^{\bullet}_{r_1}(T,v).
\]
The planar maps $\Delta$ such that $\Delta = L_{r_1,r_2}^\bullet (T,v)$ for some integers $r_2 > r_1 >0$ and some pointed triangulation of the sphere $(T,v)$ are the triangulations of the $(r,p,q)-$cylinder for some integers $p,q \geq 1$ (we will see in a moment how to canonically root the layers of a triangulation).

\bigskip

Fix $r,p,q>0$ and $\Delta$ a triangulation of the $(r,p,q)-$cylinder. The \emph{skeleton decomposition} of $\Delta$ consists of an ordered forest of $q$ rooted plane trees with maximal height $r$ and a collection of triangulations with a boundary indexed by the vertices of the forest of height stricly less than $r$.

Borrowing from Krikun \cite{Kr} and Curien and Le Gall \cite{CLGfpp}, we define the growing sequence of hulls of $\Delta$ as follows: for $1 \leq j \leq r - 1$, the ball $B_j(\Delta)$ is the union of all faces of $\Delta$ having a vertex at distance stricly smaller than $j$ from the root face, and the hull $B_j^\bullet (\Delta)$ consists of $B_j(\Delta)$ and all the connected components of its complement in $\Delta$ except the one containing the exterior face. By convention $B_r^\bullet(\Delta) = \Delta$. For every $j$, the hull $B_j^\bullet (\Delta)$ is a triangulation of the $(j,p,q')$-cylindler for some non negative integer $q'$, and we denote its exterior boundary by $\partial_j \Delta$. By convention $\partial_0 \Delta$ is the boundary of the root face of $\Delta$. In addition, every cycle $\partial_j \Delta$ is oriented so that $B_j^\bullet (\Delta)$ is always on the right hand side of $\partial_j \Delta$.

Now let $\mathcal{N}(\Delta)$ be the collection of all edges of $\Delta$ that belong to one of the cycles $\partial_i \Delta$ for some $0 \leq i \leq r$. This set is a discrete version of the metric net of the Brownian map introduced by Miller and Sheffield \cite{MS}. In order to define a genealogy on $\mathcal{N}(\Delta)$, notice that, for $1 \leq i \leq r$, every edge of $\partial_i \Delta$ belongs to exactly one face of $\Delta$ whose third vertex belongs to $\partial_{i-1} \Delta$ (it is the face on its right hand side). Such faces are usually called down triangles of height $i$. Now, for any $1 \leq i \leq r$, we say that an edge $e \in \partial_i \Delta$ is the parent of an edge $e' \in \partial_{i-1} \Delta$ if the first vertex belonging to a down triangle of height $i$ encountered when turning around the oriented cycle $\partial_{i-1} \Delta$ and starting at the end vertex of the oriented edge $e'$ belongs to the down triangle associated to $e$. See Figure \ref{fig:genealogy} for an illustration.

These relations define a forest $F$ of $q$ rooted trees, its vertices being in one-to-one correspondence with the edges of $\mathcal{N} (\Delta)$, that inherit from the planar structure of $\Delta$, making them planar rooted trees. In addition we can order canonicaly the trees, starting from the one containing the root edge of $\Delta$ and following the orientation of $\partial_r \Delta$. Notice also that every tree of the forest has height smaller than or equal to $r$ and that the whole forest has exactly $p$ vertices at height $r$.

\begin{figure}[ht!]
\begin{center}
\includegraphics[width=0.8\textwidth]{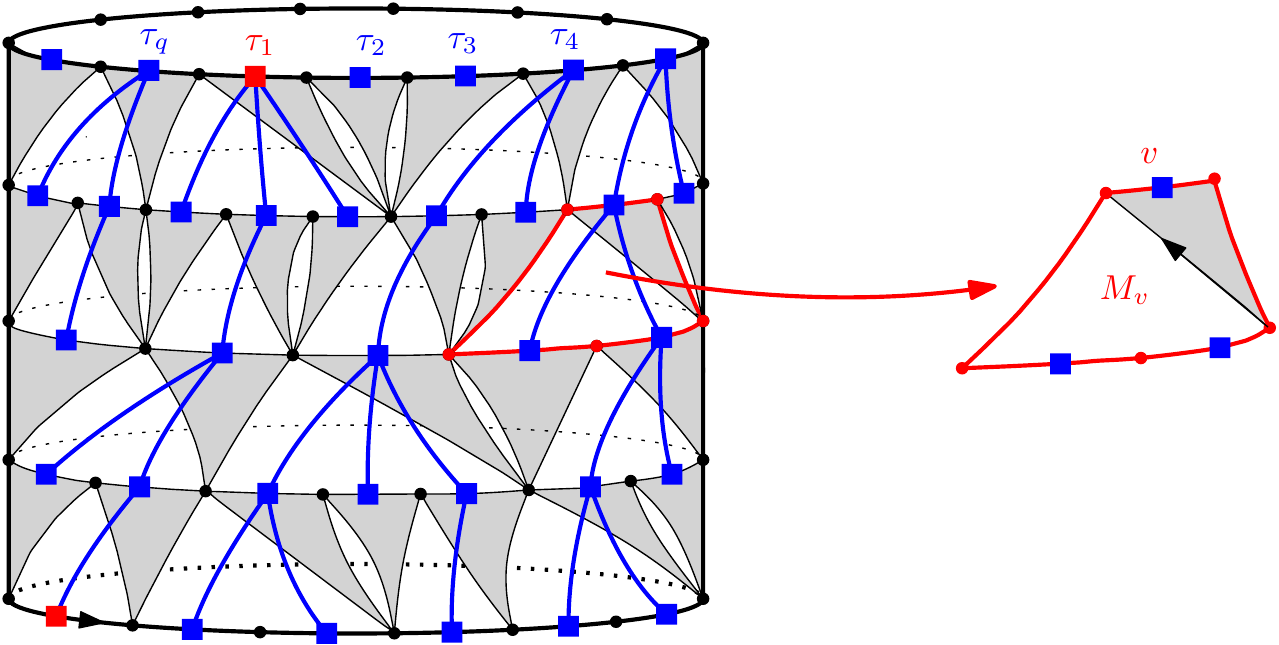}
\caption{\label{fig:genealogy} Skeleton decomposition of a triangulation of the cylinder. The distinguished vertex corresponding to the root edge of the triangulation is the red one on the bottom left. Left: construction of the forest. Right: triangulation with a boundary filling a slot.}
\end{center}
\end{figure}

To completely describe $\Delta$, in addition to the forest $F$ that gives the full structure of $\mathcal{N} (\Delta)$ and the associated down triangles, we need to specify the structure of the submaps of $\Delta$ lying in the interstices, or slots, bounded by its down triangles. More precisely, to each edge $e \in \partial_i \Delta$ where $1\leq i \leq r$, we associate a slot bounded by its children and the two edges joining the starting vertex of $e$ to $\partial_{i-1} \Delta$ (if $e$ has no child, these two edges may or may not be glued into a single edge). This slot is rooted at its unique boundary edge belonging to the down triangle associated to $e$, the orientation chosen so that the interior of the slot is on the left hand side of the root. See Figure \ref{fig:genealogy} for an illustration. With these conventions, the slot associated to an edge $e$ is filled with a well defined triangulation of the $(c_e + 2)$-gon, where $c_e$ is the number of children of $e$ in the forest $F$. The triangulation of the $(r,p,q)$-cylinder $\Delta$ is then fully characterized by the forest $F$ and the collection of triangulations with a boundary associated to the vertices of $F$ of height stricly less than $r$.

\bigskip

To summerize, let us say that a pointed forest is \emph{$(r,p,q)$-admissible} if
\begin{enumerate}
\item It is composed of an ordered sequence of $q$ rooted plane trees of height lesser than or equal to $r$.
\item It has exaclty $p$ vertices at height $r$,
\item the distinguished vertex has height $r$ and belongs to the first tree.
\end{enumerate}
We denote by $\mathcal F(r,p,q)$ the set of all $(r,p,q)$-admissible forests, and for any $F \in \mathcal{F}(r,p,q)$ we denote by $F^\star$ for the set of vertices of $F$ at height stricly smaller than $r$.

The skeleton decomposition presented above is a bijection between triangulations of the $(r,p,q)$-cylinders and pairs consisting of a $(r,p,q)$-admissible forest $F$ and a collection $(M_v)_{v \in F^\star}$, where, for each $v \in F^\star$ and denoting by $c_v$ the number of children of $v$ in $F$, $M_v$ is a triangulation of the $(c_v +2)$-gon. We say that the forest associated to a triangulation of a cylinder $\Delta$ is its skeleton and denote it by $\mathrm{Skel}(\Delta)$.

As metioned earlier, this decomposition allows to canonically root the layers of a triangulation by rooting each layer at the ancestor of the root edge of the triangulation in its skeleton.

\subsection{The UIPT and its skeleton decomposition}

Thanks to the spatial Markov property (see  \cite{AS}, Theorem 5.1), the skeleton decomposition is particularly well suited to study the UIPT. Indeed, for any intergers $r,q \geq 1$ and any $(r,1,q)$-admissible forest $F$, this property states that conditionally on the event $\{\mathrm{Skel} (B_{r}^\bullet ( \mathbf{T}_{\infty})) = F \}$, the triangulations filling the slots associated to the down triangles constitute a family of independent Boltzmann triangulations $\left(\mathbf T^{(c_v +2)}\right)_{v \in F^\star}$ where, for any integer $p \geq 1$, the law of the Boltzmann triangulation of the $p-$gon is given by
\[
\mathbb{P} \left(\mathbf T^{(p)} = T \right) = \frac{\rho^{n}}{T_p(\rho)}
\]
for any triangulation of the $p-$gon $T$ with $n$ inner vertices.

From this, a lot of information on the skeleton decomposition of the UIPT can be dug, such as the following Lemma that will be instrumental for our purpose.

\begin{lemma}[\cite{CLGfpp,Kr}]
\label{lem:lawhull}
Fix $r,p,q > 0$ and let $\Delta$ be a triangulation of the $(r,p,q)$-cylinder. The skeleton of $\Delta$ is a $(r,p,q)-$admissible forest $F \in \mathcal{F}(r,p,q)$. For each $v \in F^\star$, we denote the triangulation filling the slot associated to $v$ by $M_v$ and the number of its inner vertices by $n_v$. Then, for any $r' \geq 0$,
\begin{align*}
\mathbb{P} \left( L^{\bullet}_{r',r'+r}( \mathbf{T}_{\infty}) = \Delta 
\middle| |\partial B_{r'}^\bullet ( \mathbf{T}_{\infty})| = p
\right)
&= \frac{\alpha^q C(q)}{\alpha^p C(p)} \prod_{v \in F^\star} \alpha^{c(v) -1} \rho^{n_v + 1},\\
&= \frac{\alpha^q C(q)}{\alpha^p C(p)} \prod_{v \in F^\star} \theta(c_v) \prod_{v \in F^\star} \frac{\rho^{n_v}}{T_{c_v +2} (\rho)},
\end{align*}
where $\theta$ is the critical offspring distribution whose generating function $\varphi$ is given by
\[
\varphi (u) = \sum_{i=0}^{\infty} \theta(i) u^i 
= \frac{\rho}{\alpha^2 u} \left( T(\rho, \alpha u) - T(\rho,0) \right) 
= 1-\left( 1 + \frac{1}{\sqrt{1-u}}\right)^{-2}, \quad u \in [0,1].
\]
\end{lemma}

Lemma \ref{lem:lawhull} is not hard to establish (see \cite{CLGfpp,Kr} for the proof), and its main interest is that it allows do do exact computations by interpreting the product over vertices of the forest as the probability of some events for a branching process associated to $\varphi$. As we will do similar computations in various situations, let us give an example taken from \cite{Kr} for the sake of clarity, and because it will be needed later. Say we want to compute $\mathbb{P} \left( \left| \partial B^{\bullet}_r( \mathbf{T}_{\infty}) \right| = q \right)$ for some $q>0$. Since $\partial B_{0}^\bullet ( \mathbf{T}_{\infty})$ is the root edge of $\mathbf T _{\infty}$, which we recall is a loop, it has length $1$ and the formula of Lemma \ref{lem:lawhull} directly gives:
\[
\mathbb{P} \left( \left| \partial B^{\bullet}_r( \mathbf{T}_{\infty}) \right| = q \right) = \frac{\alpha^q C(q)}{\alpha C(1)} \sum_{F \in \mathcal{F}(r,1,q)} \prod_{v \in F^\star} \theta(c_v) = \frac{\alpha^q C(q)}{\alpha C(1)} \frac{1}{q} \sum_{F \in \mathcal{F}'(r,1,q)} \, \prod_{v \in F^\star} \theta(c_v),
\]
where $\mathcal{F}'(r,1,q)$ is the set of all ordered forests of rooted plane trees with height lesser than or equal to $r$, the whole forest having a single vertex at height $r$. Thus $\mathcal{F}'(r,1,q)$ is just the set of forests in $\mathcal{F}(r,1,q)$ up to a circular permutation, explaining the factor $1/q$. But now the quantity
\[
\sum_{F \in \mathcal{F}'(r,1,q)} \, \prod_{v \in F^\star} \theta(c_v)
\]
is exactly the probability that a Galton-Waton branching process with offspring distribution given by $\varphi$ started with $q$ particles has a single particle at generation $r$. Therefore, we have
\[
\mathbb{P} \left( \left| \partial B^{\bullet}_r( \mathbf{T}_{\infty}) \right| = q \right) = \frac{\alpha^q C(q)}{\alpha C(1)}
\frac{1}{q} \, [u](\varphi^{\{r\}}(u))^q
\]
where $\varphi^{\{ r\}} (u) = \underbrace{\varphi \circ \cdots \circ \varphi}_{r \text{ times}}(u)$ and $[u]f(u)$ is the coeffiscient in $u$ of the Taylor series at $0$ of the function $f$. The iterates $\varphi^{\{r\}}$ can be computed explicitely (see Lemma \ref{lem:ExpPhin} with $t=1$), giving the following exact formula that will be used later in the paper
\begin{equation}
\label{eq:lawPerim}
\mathbb{P} \left( \left| \partial B^{\bullet}_r( \mathbf{T}_{\infty}) \right| = q \right) = \frac{\alpha^q C(q)}{\alpha C(1)} \left(1 - \frac{1}{(r+1)^2} \right)^{q-1} \frac{1}{(r+1)^3}.
\end{equation}

\section{Hull volume}
\label{sec:hullvol}

\subsection{A branching process}

In this section, we will focus on the generating function of the number of vertices in the hulls of the UIPT. To that aim, we start with the following result:
\begin{lemma}
\label{Lem:hullvolgen}
For any integer $r > 0$, $s \in [0,1]$ and $t \in [0,1]$, one has
\[
\mathbb{E} \left[ s^{|B^{\bullet}_r(\mathbf{T}_{\infty})| } \right] =
s \sum_{q \geq 1}
\frac{(\alpha t)^q C(q)}{\alpha t C(1)}
\sum_{F \in \mathcal{F}(r,1,q)} \quad \prod_{v \in F^{\star}} \rho s \cdot (\alpha t)^{c(v) -1} \cdot T_{c(v) +2} (\rho s).
\]
\end{lemma}

\begin{remark}
This Lemma and Proposition \ref{prop:HullPhin} are in fact a consequences of Proposition \ref{prop:layervol}, but we still provide independent proofs because they provide a nice framework to introduce the functions $\varphi_t$ in \eqref{eq:exprPhit} that are central to this work.
\end{remark}

\begin{proof}[Proof of Lemma \ref{Lem:hullvolgen}]
Fix $r,q>0$ and $\Delta$ a triangulation of the $(r,1,q)$-cylinder having $F \in \mathcal{F}(r,1,q)$ as skeleton. Recall that for every $v \in F^\star$ we denote by $c_v$ the number of children of $v$ in $F$ and by $n_v$ the number of inner vertices of the triangulation of the $(c_v + 2)$-gon filling the slot associated to $v$. With these notations we have
\[
|\Delta| - 1 = \sum_{v \in F^\star} (n_v + 1),
\]
the $-1$ taking into account that the right hand side of the previous equality does not count the unique vertex of height $r$ in $F$. Lemma \ref{lem:lawhull} then gives
\[
s^{|\Delta|} \mathbb{P} \left( B^{\bullet}_R( \mathbf{T}_{\infty}) = \Delta \right)
= s \frac{\alpha^q C(q)}{\alpha C(1)} \prod_{v \in F^\star} \alpha^{c(v) -1} (s \rho)^{n_v + 1}
\]
and summing over every triangulation of the $(r,1,q)$-cylinder having $F$ as skeleton we obtain
\begin{align*}
\sum_{\Delta : \, \mathrm{Skel}(\Delta) = F} \, s^{|\Delta|} \mathbb{P} \left( B^{\bullet}_r( \mathbf{T}_{\infty}) = \Delta \right)
& = s \frac{\alpha^q C(q)}{\alpha C(1)} \prod_{v \in F^\star} \alpha^{c(v) -1} \sum_{n_v \geq 0} |\mathcal{T}_{n_v,c_v+2}| \cdot (s \rho)^{n_v + 1}\\
& = s \frac{\alpha^q C(q)}{\alpha C(1)}
\prod_{v \in F^{\star}} \rho s \cdot \alpha^{c(v) -1} \cdot T_{c(v) +2} (\rho s).
\end{align*}
Since for any $t\in [0,1]$ and any forest $F \in \mathcal{F}(r,1,q)$ one has
\[
\prod_{v \in F^{\star}} t^{c(v) -1} = t^{1-q},
\]
we can write
\begin{align*}
\sum_{\Delta : \, \mathrm{Skel}(\Delta) = F} \, s^{|\Delta|} \mathbb{P} \left( B^{\bullet}_R( \mathbf{T}_{\infty}) = \Delta \right)
& = s \frac{(\alpha t)^q C(q)}{\alpha t C(1)}
\prod_{v \in F^{\star}} \rho s \cdot (\alpha t)^{c(v) -1} \cdot T_{c(v) +2} (\rho s)
\end{align*}
and the result follows by summing over $q \geq 1$ and over every $(r,1,q)$-admissible forest.
\end{proof}

\bigskip

As was done for Lemma \ref{lem:lawhull}, we want to interpret the numbers
$\left( \rho s \cdot (\alpha t)^{i -1} \cdot T_{i +2} (\rho s) \right)_{i\geq 0}$ appearing in Lemma \ref{Lem:hullvolgen} as an offspring probability distribution. For $(s,t) \in [0,1]^2$, the generating function of these numbers is defined, for every $u\in [0,1]$, by
\begin{align*}
\Phi_{s,t}(u) &= \sum_{i \geq 0} \rho s \cdot (\alpha t)^{i -1} \cdot T_{i +2} (\rho s) u^i
=\frac{\rho s}{(\alpha t)^2 u} \left(T(\rho s, \alpha t u) - T(\rho s ,0) \right).
\end{align*}
The functions $\Phi_{s,t}$ are clearly non negative and increasing, thus we just have to pick $(s,t)$ such that $\Phi_{s,t}(1) = 1$. Using formulas \eqref{eq:expT0h} and \eqref{eq:expTxy}, simple computations yield
\[
\Phi_{s,t}(1) = \frac{
6 t \cdot \alpha^{-1} h (\rho s)
- 3 t \cdot \alpha^{-2} h^2 (\rho s) - 2 s^2
+2s\sqrt{s^2 + 3 t^2 -t^3+ 3 t \cdot \alpha^{-2} h^2 (\rho s) - 6 t \cdot \alpha^{-1} h (\rho s)}
}{t^3}.
\]
To solve this equation, we first notice that, from equation \eqref{eq:h} satisfied by $h$, we have
\[
s^2 = \alpha^{-2} h^2(\rho s) \left(3 - 2 \alpha^{-1} h(\rho s) \right).
\]
This suggests to consider $t(s) \in [0,1]$ such that
\[
h(\rho s) = \alpha t(s),
\]
or equivalently with equation \eqref{eq:h},
\[
s = t(s)^2 \left(3 - 2 t(s) \right).
\]
This parametrization yields
\begin{align*}
\Phi_{s,t(s)}(1) & = \frac{
6  t^2
- 3 t^3 - 2 s^2
+2s\sqrt{s^2 - 3 t^2 + 2 t^3}
}{t^3} =1.
\end{align*}

From now on, we will only consider pairs $(s,t) \in [0,1]^2$ such that $s = t \sqrt{3-2t}$. For such pairs we define, for every $u\in[0,1]$,
\begin{align*}
\varphi_t(u) & := \Phi_{s,t(s)}(u) \\
& = \frac{
6  u t^2 - 3 u t^3 - 2 s^2
+2s\sqrt{s^2 + 3 t^2 u^2 -t^3 u^3 + 3 t^3 u- 6 t^2 u}
}{t^3 u^2}
\end{align*}
which is the generating function of a probability distribution. Simple computations give the following alternative expression:
\begin{equation}
\label{eq:exprPhit}
\varphi_t(u)= 1 - \left( \frac{1}{\sqrt{1-u}} \sqrt{\frac{3 -2t}{t}} + \sqrt{1 + \frac{3}{1-u} \left(\frac{1 -t}{t} \right)}\right)^{-2}.
\end{equation}
This expression is not unlike the expression of $\varphi$ given in Lemma \ref{lem:lawhull}, and $\varphi_1 = \varphi$ which is no surprise.

\bigskip

The next result gives an expression of the generating function of the volume of hulls of the UIPT in terms of iterates of the functions $\varphi_t$. We will see in its proof that it takes advantage of the branching process associated to $\varphi_t$.

\begin{proposition}
\label{prop:HullPhin}
Fix $r > 0$ and a pair $(s,t) \in [0,1]^2$ such that $s=t\sqrt{3-2t}$, then
\[
\mathbb{E} \left[ s^{|B^{\bullet}_r(\mathbf{T}_{\infty})| } \right] =
s
\left( 1 - t \varphi_t^{\{ r\}} (0) \right)^{-3/2} 
\varphi_t^{\{ r\}'}(0)
\]
where $\varphi_t^{\{ r\}} (u) = \underbrace{\varphi_t \circ \cdots \circ \varphi_t}_{r \text{ times}}(u)$. 
\end{proposition}

\begin{proof}
First, we interpret the sum over forests in $\mathcal{F}(r,1,q)$ appearing in the statement of Lemma \ref{lem:lawhull} as the probability of an event for a branching process with offspring distribution given by $\varphi_t$. To do that we first write
\begin{align}
\sum_{F \in \mathcal{F}(r,1,q)} \quad \prod_{v \in F^\star} \rho s \cdot (\alpha t)^{c(v) -1} \cdot T_{c(v) +2} (\rho s) &= \sum_{F \in \mathcal{F}(r,1,q)} \quad \prod_{v \in F^\star} [u^{c_v}] \varphi_t (u) \notag\\
& = \frac{1}{q} \sum_{F \in \mathcal{F}'(r,1,q)} \quad \prod_{v \in F^\star} [u^{c_v}] \varphi_t (u) \label{eq:BPforest}
\end{align}
where $\mathcal{F}'(r,1,q)$ is the set of all ordered forests of $q$ rooted plane trees of height lesser than or equal to $r$ and having exactly one vertex at height $r$. The forests in $\mathcal{F}'(r,1,q)$ are obtained from the forests in $\mathcal{F}(r,1,q)$ by a circular permutation of the order of their trees, so that the vertex at height $r$ does not necessarily belong to the first tree, explaning the factor $\frac{1}{q}$. But now, the right hand side of \eqref{eq:BPforest} without this factor $\frac{1}{q}$ is exactly the probability that a Galton-Watson branching process with offspring distribution given by $\varphi_t$ started with $q$ particles has exaclty one particle at generation $r$. This probability is $[u] \left(\varphi_t^{\{ r\}} (u) \right)^q$ and thus
\begin{align*}
\mathbb{E} \left[ s^{|B^{\bullet}_r(\mathbf{T}_{\infty})|} \right] & =
s\sum_{q \geq 1} \,
\frac{(\alpha t)^q C(q)}{\alpha t C(1)} \,
\frac{1}{q} \, [u] \left(\varphi_t^{\{ r\}} (u) \right)^q,\\
& = s \, [u] \frac{1}{ 6 \alpha t} \,
\sum_{q\geq 1} \, \binom{2q}{q} \left(3 \alpha t \varphi_t^{\{ r\}} (u) \right)^q,\\
& = s \, [u] \frac{\left(1-4 \cdot 3 \alpha t \varphi_t^{\{ r\}} (u)\right)^{-1/2} -1}{6 \alpha t},\\
& = s \, [u] \frac{2}{t} \left(\left(1-t \varphi_t^{\{ r\}} (u)\right)^{-1/2} -1 \right),\\
& = s\left( 1 - t \varphi_t^{\{ r\}} (0) \right)^{-3/2} 
[u]\varphi_t^{\{ r\}}(u),
\end{align*}
giving the result.
\end{proof}

\subsection{Explicit computations and proof of Theorem \ref{th:volhull}}

Before proving Theorem \ref{th:volhull}, let us first compute explicitely the iterates $\varphi_t^{\{ r\}}$ appearing in Proposition \ref{prop:HullPhin}:
\begin{lemma}
\label{lem:ExpPhin}
Fix $t\in[0,1[$ and $r\in \mathbb{N}$, then, for every $u\in[0,1]$,
\[
\varphi_t^{\{ r\}} (u) = 1 - 3 \frac{1-t}{t \, \left( \sinh \left( \sinh^{-1} \left( \sqrt{\frac{3(1-t)}{t(1-u)}}\right) + r \cosh^{-1} \left(\sqrt{\frac{3-2t}{t}} \right) \right) \right)^2}
\]
and
\[
\varphi_1^{\{ r\}} (u) = 1 - \frac{1}{\left( \frac{1}{\sqrt{1-u}}+ r \right)^2}.
\]
\end{lemma}
\begin{proof}
Fix $t,u \in [0,1]$ and, for every $n\in \mathbb{N}$, denote
\[
v_n = \frac{1}{\sqrt{1 - \varphi_t^{\{ n\}} (u)}}.
\]
From the expression of $\varphi_t$ given in equation \eqref{eq:exprPhit} we deduce that the sequence $(v_n)_{n \geq 0}$ satisfies
\begin{equation}
\label{eq:defvn}
\begin{cases}
v_0 = \frac{1}{\sqrt{1 -u}}\\
v_{n+1} = a v_n + \sqrt{1+(a^2 -1) v_n^2}
\end{cases}
\end{equation}
with 
\[
a = \sqrt{\frac{3-2t}{t}} \geq 1.
\]
If $a=1$, the sequence $(v_n)$ has arithmetic progression and the result is trivial. Therefore we suppose $t <1$, and thus $a >1$. Define $w_n > 0$ such that
\[
\sinh (w_n) = \sqrt{a^2 -1} \, v_n;
\]
then the recursion relation \eqref{eq:defvn} satisfied by $(v_n)_{n \geq 0}$ becomes
\begin{align*}
\sinh(w_{n+1}) &= a \sinh(w_n) + \sqrt{a^2 -1} \cosh(w_n)\\
& = \sinh \left( w_n + \cosh^{-1}(a) \right).
\end{align*}
This shows that the sequence $(w_n)$ has arithmetic progression and we have, for every $n \geq 0$,
\begin{align*}
w_n &= \sinh^{-1} \left( \sqrt{a^2 -1} \, v_0 \right) + n \cosh^{-1}(a)
\end{align*}
and the result follows easily.
\end{proof}

\begin{remark}
The sequence $(v_n)$ defined by \eqref{eq:defvn} satisfies the following second order linear recursion
\[
v_{n+1} = 2a v_n - v_{n-1}
\]
that can be derived directly from \eqref{eq:defvn} by noticing that
\[
v_{n+1}^2 - 2 a v_n v_{n+1} + v_n^2 = 1
\]
yielding
\[
0 = v_{n+1}^2 - 2 a v_n (v_{n+1} - v_{n-1}) - v_{n-1}^2 = (v_{n+1} - v_{n-1})(v_{n+1} - 2 a v_n + v_{n-1}).
\]
This gives an alternate derivation of $v_n$ where hyperbolic functions do not appear directly.
\end{remark}

\begin{proof}[Proof of Theorem \ref{th:volhull}]
Theorem \ref{th:volhull} is now a direct consequence of Proposition \ref{prop:HullPhin} and Lemma \ref{lem:ExpPhin}. Indeed we have from Lemma \ref{lem:ExpPhin}
\[
\varphi_t^{\{ r\}} (0) = 1 - 3 \frac{1-t}{t \, \left( \sinh \left( \sinh^{-1} \left( \sqrt{3\frac{1-t}{t}}\right) + r \cosh^{-1} \left(\sqrt{\frac{3-2t}{t}} \right) \right) \right)^2}
\]
and
\begin{align*}
[u] \varphi_t^{\{ r\}} (u) &=\left(3 \frac{1-t}{t} \right)^{3/2} \left( \frac{3-2t}{t} \right) ^{-1/2}
\left( \sinh \left( \sinh^{-1} \left( \sqrt{3\frac{1-t}{t}}\right) + r \cosh^{-1} \left(\sqrt{\frac{3-2t}{t}} \right) \right) \right)^{-3}\\
& \quad \times \cosh \left( \sinh^{-1} \left( \sqrt{3\frac{1-t}{t}}\right) + r \cosh^{-1} \left(\sqrt{\frac{3-2t}{t}} \right) \right).
\end{align*}
The result then follows from Proposition \ref{prop:HullPhin}.
\end{proof}

The proof of Corollary \ref{cor:hullcond} relies on Proposition \ref{prop:layervol} proved in the next Section but we give it here since it is more in the spirit of this section.

\begin{proof}[Proof of Corollary \ref{cor:hullcond}]
Proposition \ref{prop:layervol} gives with $r' = 0$ and $p=1$:
\[
\mathbb E \left[ s^{|B^{\bullet}_r(\mathbf{T}_{\infty})|} \middle| |\partial B^{\bullet}_r(\mathbf{T}_{\infty})| = q \right] = s t^{q-1} 
\frac{\left( \varphi_t^{\{ r\}} (0) \right)^{q-1} \varphi_t^{\{ r\} \prime} (0)}{\left( \varphi_1^{\{ r\}} (0) \right)^{q-1} \varphi_1^{\{ r\} \prime} (0)}.
\]
Putting $s = e^{-\lambda/R^4}$, $r = \lfloor x R \rfloor$ and $q = \lfloor \ell R^2 \rfloor$ we have the following asymptotics:
\begin{align*}
t^{q} &= 1- \frac{\sqrt{2\lambda/3}}{R^2} + \mathcal{O} \left( \frac{1}{R^4} \right),\\
\varphi_{t}^{\lfloor x R \rfloor} (0) 
& = 1 - \frac{\sqrt{6 \lambda}}{R^2}
\left( 
\sinh \left( (6 \lambda )^{1/4} x \right)
\right)^{-2} + \mathcal{O} \left( \frac{1}{R^4} \right),\\
\varphi_{1}^{\lfloor x R \rfloor} (0) 
& = 1 - \frac{1}{(\lfloor x R \rfloor)^2},\\
\varphi_{t}^{\lfloor x R \rfloor '} (0)
& \sim
\left( \frac{\sqrt{6 \lambda}}{R^2} \right)^{3/2} \frac{\cosh \left( (6 \lambda )^{1/4} x \right)}{\left( \sinh \left( (6 \lambda)^{1/4} x \right) \right)^3},\\
\varphi_{1}^{\lfloor x R \rfloor '} (0)
& \sim \frac{1}{(x R)^3},
\end{align*}
and the result follows easily.
\end{proof}

\section{Hull volume process}
\label{sec:hullproc}

In order to prove Theorem \ref{th:hullproc}, we first compute the generating function of the volume of layers of the UIPT:

\begin{proposition}
\label{prop:layervol}
Let $r,r',p,q$ be nonegative integers and $(s,t) \in [0,1]$ such that $s = t \sqrt{3-2t}$, then
\[
\mathbb{E} \left[ s^{|L^{\bullet}_{r',r'+r}(\mathbf{T}_{\infty})|} \middle| |\partial B^{\bullet}_{r'+r}(\mathbf{T}_{\infty})| = q, |\partial B^{\bullet}_{r'}(\mathbf{T}_{\infty})| = p \right] = s^p t^{q-p} \frac{[u^p] \left( \varphi_t^{\{r\}}(u) \right)^q}{[u^p] \left( \varphi_1^{\{r\}}(u) \right)^q}.
\]
\end{proposition}

\begin{proof}
The proof of this Proposition is very much in the spirit of the proofs of Lemma \ref{Lem:hullvolgen} and Proposition \ref{prop:HullPhin}. Indeed, let $\Delta$ be a triangulation of the $(r,p,q)-$cylinder having $F \in \mathcal F (r,p,q)$ as skeleton. We have
\[
|\Delta| - p = \sum_{v \in F^\star} (n_v +1),
\]
giving, with Lemma \ref{lem:lawhull} and summing over every triangulation having $F$ as skeleton,
\begin{align*}
\sum_{\Delta : \, \mathrm{Skel}(\Delta) = F}  \, s^{|\Delta|} \mathbb{P} & \left( L^{\bullet}_{r',r'+ r}( \mathbf{T}_{\infty}) = \Delta 
\middle| |\partial B_{r'}^\bullet ( \mathbf{T}_{\infty})| = p
\right)\\
& \quad = s^p \frac{\alpha^q C(q)}{\alpha^p C(p)} \prod_{v \in F^\star} \alpha^{c(v) -1} \sum_{n_v \geq 0} |\mathcal{T}_{n_v,c_v+2}| (s \rho)^{n_v + 1}\\
& \quad = s^p \frac{\alpha^q C(q)}{\alpha^p C(p)}
\prod_{v \in F^{\star}} \rho s \cdot \alpha^{c(v) -1} \cdot T_{c(v) +2} (\rho s)\\
& \quad = s^p \frac{(\alpha t)^q C(q)}{(\alpha t)^p C(p)}
\prod_{v \in F^{\star}} [u^{c_v}]\varphi_t(u).
\end{align*}
Summing over every $(r,p,q)$-admissible forest then gives
\[
\mathbb{E} \left[ s^{|L^{\bullet}_{r',r'+r}(\mathbf{T}_{\infty})|}
\mathbf{1}_{ \big\{ |\partial B^{\bullet}_{r'+r}(\mathbf{T}_{\infty})| = q \big\}}
\middle|
|\partial B^{\bullet}_{r'}(\mathbf{T}_{\infty})| = p \right]
= s^p \frac{(\alpha t)^q C(q)}{(\alpha t)^p C(p)} 
\sum_{F \in \mathcal{F}(r,p,q)} \, \prod_{v \in F^\star} [u^{c_v}] \varphi_t (u).
\]
Now, if $\mathcal{F'}(r,p,q)$ denotes the set of all $(r,p,q)$-admissible forests up to a cyclic permutation of the order of the trees, each tree in $\mathcal{F}(r,p,q)$ corresponds to exactly $q$ trees of $\mathcal{F'}(r,p,q)$, therefore
\[
\mathbb{E} \left[ s^{|L^{\bullet}_{r',r'+r}(\mathbf{T}_{\infty})|}
\mathbf{1}_{ \big\{ |\partial B^{\bullet}_{r'+r}(\mathbf{T}_{\infty})| = q \big\}}
\middle|
|\partial B^{\bullet}_{r'}(\mathbf{T}_{\infty})| = p \right]
= s^p \frac{(\alpha t)^q C(q)}{(\alpha t)^p C(p)} \frac{1}{q}
\sum_{F \in \mathcal{F'}(r,p,q)} \, \prod_{v \in F^\star} [u^{c_v}] \varphi_t (u).
\]
The trees in $\mathcal{F'}(r,p,q)$ have a distinguished vertex at height $r$, and if $\mathcal{F''}(r,p,q)$ denotes the set of all rooted forests of height lesser than or equal to $r$, with $q$ trees, and having a total number $p$ of vertices at height $r$, each forest of $\mathcal{F''}(r,p,q)$ corresponds to exactly $p$ forests in $\mathcal{F'}(r,p,q)$, thus
\[
\mathbb{E} \left[ s^{|L^{\bullet}_{r',r'+r}(\mathbf{T}_{\infty})|}
\mathbf{1}_{ \big\{ |\partial B^{\bullet}_{r'+r}(\mathbf{T}_{\infty})| = q \big\}}
\middle|
|\partial B^{\bullet}_{r'}(\mathbf{T}_{\infty})| = p \right]
= s^p \frac{(\alpha t)^q C(q)}{(\alpha t)^p C(p)} \frac{p}{q}
\sum_{F \in \mathcal{F''}(r,p,q)} \, \prod_{v \in F^\star} [u^{c_v}] \varphi_t (u).
\]
But now the sum
\[
\sum_{F \in \mathcal{F''}(r,p,q)} \, \prod_{v \in F^\star} [u^{c_v}] \varphi_t (u)
\]
is the probability that a Galton-Watson process with offspring distribution given by $\varphi_t$ started with $q$ particles has $p$ particles at generation $r$. This yields
\[
\mathbb{E} \left[ s^{|L^{\bullet}_{r',r'+r}(\mathbf{T}_{\infty})|}
\mathbf{1}_{ \big\{ |\partial B^{\bullet}_{r'+r}(\mathbf{T}_{\infty})| = q \big\}}
\middle|
|\partial B^{\bullet}_{r'}(\mathbf{T}_{\infty})| = p \right]
= s^p \frac{(\alpha t)^q C(q)}{(\alpha t)^p C(p)} \frac{p}{q}
\, [u^p] \left( \varphi_t^{\{r\}} (u) \right)^q.
\]
Using the same reasoning, we can easily get
\[
\mathbb{P} \left( 
\big\{ |\partial B^{\bullet}_{r'+r}(\mathbf{T}_{\infty})| = q \big\}
\middle|
|\partial B^{\bullet}_{r'}(\mathbf{T}_{\infty})| = p
\right)
= \frac{\alpha^q C(q)}{\alpha^p C(p)} \frac{p}{q}
\, [u^p] \left( \varphi_1^{\{r\}} (u) \right)^q
\]
and the result follows.
\end{proof}

As we will see in the proof of Theorem \ref{th:hullproc}, the jumps of the process of hull perimeters will induce jumps for the process of hull volumes. This motivates the following technical result, which is a consequence of Proposition \ref{prop:layervol}, and will be used in the proof of Therem \ref{th:hullproc}.

\begin{corollary}
\label{cor:hulldiff}
Fix an integer $r > 0$ and $\ell > \delta > 0 $. Let $(p_n,q_n)_{n\geq 0}$ be non negative integers such that $n^{-2} p_n \to \ell$ and $n^{-2} q_n \to \ell - \delta$ as $n \to \infty$.
Then, conditionally on the events
\[
\big\{ |\partial B^{\bullet}_{r+1}(\mathbf{T}_{\infty})| = q_n \big\}
\cap
\big\{ |\partial B^{\bullet}_{r}(\mathbf{T}_{\infty})| = p_n \big\},
\]
the following convergence in distribution holds
\[
n^{-4} |B^{\bullet}_{r+1}(\mathbf{T}_{\infty}) \setminus B^{\bullet}_{r}(\mathbf{T}_{\infty})| \xrightarrow[n \to \infty]{(d)} \frac{4}{3} \delta^2 \cdot \xi
\]
where $\xi$ is a random variable with density $\frac{1}{\sqrt{2 \pi x^5}} e^{-\frac{1}{2x}} \mathbf{1}_{\{x >0\}}$.
\end{corollary}

\begin{proof}
Proposition \ref{prop:layervol} gives, for any $(s,t) \in [0,1]^2$ with $s=t\sqrt{3-2t}$,
\[
\mathbb{E} \left[ s^{|B^{\bullet}_{r+1}(\mathbf{T}_{\infty}) \setminus B^{\bullet}_{r}(\mathbf{T}_{\infty})|} \middle| |\partial B^{\bullet}_{r+1}(\mathbf{T}_{\infty})| = q_n, |\partial B^{\bullet}_{r}(\mathbf{T}_{\infty})| = p_n \right] = t^{q_n-p_n} \frac{[u^{p_n}] \left( \varphi_t^{\{r\}}(u) \right)^{q_n}}{[u^{p_n}] \left( \varphi_1^{\{r\}}(u) \right)^{q_n}}.
\]
We can study the asymptotic behavior of the quantity $[u^{p_n}] \left( \varphi_t(u) \right)^{q_n}$ with standart analytic techniques:
\[
[u^{p_n}] \left( \varphi_t(u) \right)^{q_n} = \frac{1}{2i\pi} \oint_{\gamma} \frac{\varphi_t(z)^{q_n}}{z^{p_n+1}} dz
\]
where $\gamma$ is a small enough contour enclosing the origin. The function $\varphi_t$ being analytic in $\mathbb C \setminus [1,+\infty[$, it is possible to deform the contour $\gamma$ into a Henkel-type contour $\gamma_n$ without changing the value of the integral (the modulus of the integrand decreases exponentially fast for $|z|$ large). For $n \geq 1$, we can take $\gamma_n$ to be the reunion on the semi infinite line $-i/n + [1, +\infty[$, oriented from right to left, the semi circle $1 + \frac{1}{n} e^{i]\pi/2,3\pi/2[}$ oriented clockwise, and the semi infinite line $+ i/n + [1, +\infty[$ oriented from left to right (see Figure \ref{fig:contours} for an illustration). The change of variable $z \to 1 + z/p_n$ then gives
\[
[u^{p_n}] \left( \varphi_t(u) \right)^{q_n} = \frac{1}{2i\pi} \oint_{\gamma_{p_n}} \frac{\varphi_t(z)^{q_n}}{z^{p_n+1}} dz = \frac{1}{2i \pi p_n} \oint_{\mathcal H} \frac{\varphi_t(1 + z/p_n)^{q_n}}{(1+z/p_n)^{p_n+1}} dz
\]
where $\mathcal H$ is the Henkel contour, that is the reunion of the semi infinite line $-i + [0, +\infty[$, oriented from right to left, the semi circle $ e^{i]\pi/2,3\pi/2[}$ oriented clockwise, and the semi infinite line $ i + [0, +\infty[$ oriented from left to right (see Figure \ref{fig:contours} for an illustration).

\begin{figure}[ht!]
\begin{center}
\includegraphics[width=0.8\textwidth]{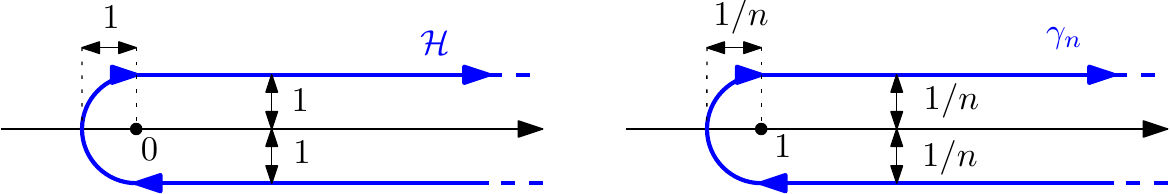}
\caption{\label{fig:contours} The contours $\mathcal H$ and $\gamma_n$.}
\end{center}
\end{figure}

\bigskip

From equation \eqref{eq:exprPhit} we have for $z \in \mathbb C \setminus [1, +\infty [$:
\[
\varphi_t(z) = 1 - \frac{t}{3-2t} (1-z) \left( 1+ \sqrt{1-z \frac{t}{3-2t}}\right)^{-2}.
\]
If $s_n = e^{-\lambda/n^4}$, then $t_n = 1 - \frac{\sqrt{2 \lambda/ 3}}{n^2} + \mathcal O (n^{-4})$, thus for $z \in \mathcal{H}$:
\begin{align*}
\varphi_{t_n} (1 + z/p_n) & = 1 + \frac{z}{p_n} \left( 1 - \frac{\sqrt{6 \lambda}}{n^2}+ \mathcal O (n^{-4}) \right) \cdot \left( 1 + \sqrt{ -\frac{z}{p_n} + \frac{\sqrt{6 \lambda}}{n^2}+ \mathcal O (n^{-4}) } \right)^{-2}, \\
& =  1+ \frac{z}{p_n} + \frac{2 z \left(\ell \sqrt{6 \lambda} - z\right)^{1/2}}{p_n^{3/2}} + \mathcal O (p_n^{-2}).
\end{align*}
Then we have, for $z \in \mathcal H$,
\[
\varphi_{t_n} (1 + z/p_n)^{q_n} = e^{z(1-\delta/\ell)} \left(1 + 2 (1-\delta/\ell) z \left(\ell \sqrt{6 \lambda} - z\right)^{1/2} p_n^{-1/2} + \mathcal{O} (p_n^{-1} )\right),
\]
giving
\[
\frac{\varphi_{t_n}(1 + z/p_n)^{q_n}}{(1+z/p_n)^{p_n+1}} = e^{- z \delta/\ell} \left(1 + 2 (1-\delta/\ell) z \left(\ell \sqrt{6 \lambda} - z\right)^{1/2} p_n^{-1/2} + \mathcal{O} (p_n^{-1} )\right).
\]
Since for any $\alpha \in \mathbb R$
\[
\frac{1}{2i\pi}\oint_{\mathcal{H}} (-z)^{\alpha} e^{-z} dz = \frac{1}{\Gamma(-\alpha)},
\]
we get, at least on a formal level,
\begin{align}
\label{eq:formalsingularity}
[u^{p_n}] \left( \varphi_t(u) \right)^{q_n} &= \frac{2 (1-\delta/\ell)}{2i\pi p_n^{3/2}} \oint_{\mathcal{H}} e^{- z \delta/\ell} z (\ell \sqrt{6\lambda} - z)^{1/2} dz + \mathcal{O}({p_n}^{-2}), \notag\\
&= \frac{2 (1-\delta/\ell)}{2i\pi p_n^{3/2} \delta/\ell} e^{- \delta \sqrt{6\lambda}}
 \left( \frac{\ell}{\delta} \right)^{3/2} \oint_{\mathcal{H}} e^{- z} \left( \delta \sqrt{6\lambda} + z \right) \left(-z \right)^{1/2} dz+ \mathcal{O}(p_n^{-2}) \notag,\\
&= \frac{2 (1-\delta/\ell)}{p_n^{3/2}} \left( \frac{\ell}{\delta} \right)^{5/2} e^{- \delta \sqrt{6\lambda}}
\left( \frac{\delta \sqrt{6\lambda}}{\Gamma(-1/2)} + \frac{-1}{\Gamma(-3/2)} \right) + \mathcal{O}(p_n^{-2})\notag,\\
& = \frac{2 (1-\delta/\ell)}{p_n^{3/2}} \left( \frac{\ell}{\delta} \right)^{5/2} e^{- \delta \sqrt{6\lambda}}
\left( \frac{- \delta \sqrt{6\lambda}}{\frac{3}{2}\Gamma(-3/2)} + \frac{-1}{\Gamma(-3/2)} \right) + \mathcal{O}(p_n^{-2}).
\end{align}
The justification of the formal argument used to derive \eqref{eq:formalsingularity} is quite standart in analytic combinatorics. For example it is identical to the one done in the proof of Theorem VI.1 of \cite{FS}.

This asymptotic expansion yields
\[
\frac{[u^{p_n}] \left( \varphi_{t_n}(u) \right)^{q_n}}{[u^{p_n}] \left( \varphi_1(u) \right)^{q_n}} \xrightarrow[n\to \infty]{} e^{- \delta \sqrt{6\lambda}}
\left(\frac{2}{3} \delta \sqrt{6\lambda} + 1 \right),
\]
and
\begin{align*}
\mathbb{E} & \left[ \exp \left( - \frac{\lambda}{n^4} {|B^{\bullet}_{r+1}(\mathbf{T}_{\infty}) \setminus B^{\bullet}_{r}(\mathbf{T}_{\infty})|} \right) \middle| |\partial B^{\bullet}_{r+1}(\mathbf{T}_{\infty})| = q_n, |\partial B^{\bullet}_{r}(\mathbf{T}_{\infty})| = p_n \right] \\
& \quad \quad =
t_n^{q_n-p_n} \frac{[u^{p_n}] \left( \varphi_{t_n}(u) \right)^{q_n}}{[u^{p_n}] \left( \varphi_1(u) \right)^{q_n}} \xrightarrow[n\to \infty]{} e^{- \frac{2}{3} \delta \sqrt{6\lambda}}
\left(\frac{2}{3} \delta \sqrt{6\lambda} + 1 \right),
\end{align*}
finally giving the result since the Laplace transform of $\xi$ is given by
\[
\mathbb E \left[ e^{- \lambda \xi} \right] = (1 + \sqrt{2 \lambda}) e^{-2 \lambda}
\]
for every $\lambda > 0$.
\end{proof}

We are now ready to prove Theorem \ref{th:hullproc}.
\begin{proof}[Proof of Theorem \ref{th:hullproc}]
With the help of Corollary \ref{cor:hulldiff}, the proof of this result is similar to the proof of Theorem 1 in \cite{CLG}. First, we can restrict the time interval to $[0,1]$ and verify that
\begin{equation*}
\left(R^{-2} |\partial B^{\bullet}_{\lfloor x R \rfloor}(\mathbf{T}_{\infty})|, R^{-4}|B^{\bullet}_{\lfloor x R \rfloor}(\mathbf{T}_{\infty})| \right)_{x \in [0,1]}
\xrightarrow[R \to \infty]{(d)}
\left(3^2 \cdot \mathcal L_x,4 \cdot 3^3 \cdot \mathcal M_x\right)_{x \in [0,1]}.
\end{equation*}
The convergence of the first component 
\begin{equation}
\label{eq:cvperim}
\left(R^{-2} |\partial B^{\bullet}_{\lfloor x R \rfloor}(\mathbf{T}_{\infty})|\right)_{x \in [0,1]}
\xrightarrow[R \to \infty]{(d)}
\left(3^2 \cdot \mathcal L_x\right)_{x \in [0,1]}.
\end{equation}
is already proved in \cite{Kr} via the skeleton decomposition and in \cite{CLG} via the peeling process. Therefore, we will study the second component given the first one.

For every $r \geq 1$ we can write
\[
V_r := |B^{\bullet}_{r}(\mathbf{T}_{\infty})| = 1 + \sum_{i=1}^r U_i
\]
where, for every $i \geq 1$,
\[
U_i = |B^{\bullet}_{i}(\mathbf{T}_{\infty}) \setminus B^{\bullet}_{i-1}(\mathbf{T}_{\infty})|.
\]
Fix $\varepsilon > 0$ and $R > 0$, Corollary \ref{cor:hulldiff} suggests to introduce, for $r \in \{ 1, \ldots ,R \}$,
\[
V_r^{> \varepsilon} = \sum_{i=1}^r U_i \mathbf{1}_{\left\{ P_i < P_{i-1} - \varepsilon R^2 \right\}}, \quad V_r^{\leq \varepsilon} = \sum_{i=1}^r U_i \mathbf{1}_{\left\{ P_i \geq P_{i-1} - \varepsilon R^2 \right\}},
\]
where $P_i = |\partial B_i^{\bullet} (\mathbf{T}_{\infty})|$ for every $i \geq 1$.

\bigskip

Let us first show that $R^{-4} V_R^{\leq \varepsilon}$ is small uniformly in $R$ when $\varepsilon$ is small. We will proceed with a first moment argument, and a first step is to give a bound on the expectation of $U_i$ conditionnaly on the event $\{P_{i-1} =p\}$, for $i \geq 1$. Fix $p,q \geq 1$ and let $F$ be a $(1,p,q)$-admissible forest. Recall that the spatial Markov property of the UIPT states that, conditionnaly on the event $\{\mathrm{Skel}(L^{\bullet}_{i-1,i} (\mathbf{T}_\infty)) = F\}$, the layer $L^{\bullet}_{i-1,i} (\mathbf{T}_\infty)$ is composed of its down triangles and a collection of indenpendent Boltzmann triangulations $\left(\mathbf T^{(c_v +2)}\right)_{v \in F^\star}$. There exists a universal constant $C >0$ such that, for any integer $p \geq 1$, one has
\[
\mathbb E \left[ |\mathbf T^{(p)}| \right] \leq C \, p^2
\]
(see for example \cite{CLG}, Proposition 8) and therefore
\[
\mathbb E \left[ U_i \middle| \mathrm{Skel}(L^{\bullet}_{i-1,i} (\mathbf{T}_\infty)) = F\right] \leq  C \sum_{v \in F^\star} (c_v + 2)^2 = C \sum_{v \in F^\star} c_v^2 + 4C \, q +2 \, C \, p.
\]
Using Lemma \ref{lem:lawhull}, we get, for every $i \geq 2$ and $p \geq 1$,
\begin{align*}
\mathbb E \left[ U_i \, \mathbf{1}_{\left\{P_i = q\right\}} \middle| P_{i-1} = p \right] &=
\sum_{F \in \mathcal{F}(1,p,q)} \mathbb E \left[ U_i \middle| \mathrm{Skel}(L^{\bullet}_{i-1,i} (\mathbf{T}_\infty)) = F\right] \cdot \frac{\alpha^q C(q)}{\alpha^p C(p)}  \prod_{v \in F^\star} [u^{c_v}] \varphi\\
& \leq  \frac{\alpha^q C(q)}{\alpha^p C(p)} \frac{p}{q} \left( \left( 4C\, q +2 \, C \, p \right) [u^p] \varphi^q
+ C \sum_{n_1 + \cdots n_q = p} \, \sum_{i=1}^q n_i^2 \prod_{i=1}^q [u^{n_i}] \varphi \right).
\end{align*}
The sum on the right hand side of the last equation is exactly
\[
\mathbb E \left[ (N_1^2 + \cdots + N_q^2) \mathbf{1}_{\left\{ N_1 + \cdots + N_q =p \right\}}  \right] = q \mathbb E \left[ N_1^2 \mathbf{1}_{\left\{ N_1 + \cdots +N_q =p \right\}}  \right],
\]
where $N_1, \ldots , N_q$ are independent random variables distributed according to $\varphi$. We have
\[
\mathbb E \left[ N_1^2 \mathbf{1}_{\left\{ N_1 + \cdots +N_q =p \right\}} \right] = [u^{p-2}] \left( \varphi'' \varphi^{q-1} \right) + \frac{p}{q} [u^p] \varphi^q
\]
for every $p\geq 2$, yiedling
\begin{equation}
\label{eq:espU}
\mathbb E \left[ U_i \, \mathbf{1}_{\left\{P_i = q\right\}} \middle| P_{i-1} = p \right]
\leq 
\frac{\alpha^q C(q)}{\alpha^p C(p)} \frac{p}{q} \left(
\left( 4C \, q + 3 \, C \, p \right) [u^p] \varphi^q
+ C \cdot q [u^{p-2}] \left( \varphi'' \varphi^{q-1} \right)
\right).
\end{equation}

Using
\[
\sum_{q \geq 1} \frac{1}{q} \alpha^q C(q) u^q = C' \cdot \left( \left(1 - u \right)^{-1/2} - 1 \right)
\]
for some $C' > 0$ we get
\[
\mathbb E \left[ U_i \middle| P_{i-1} = p \right]
\leq 
\frac{p}{\alpha^p C(p)} \left(
C_1 [u^p] \left( \frac{\varphi}{(1-\varphi)^{3/2}} \right) + C_2 \cdot p [u^p] \left(1 - \varphi \right)^{-1/2}  + C_3 [u^{p-2}] \left( \frac{\varphi''}{(1-\varphi)^{3/2}}\right)
\right)
\]
for every $p \geq 1$, where $C_1,C_2,C_3 > 0$ are fixed.
Using the fact that $1 - \varphi \sim 1-u$ and $\varphi'' \sim \frac{3}{2} (1-u)^{-1/2}$ as $u \to 1$, it is easy to see that
\begin{equation}
\label{eq:EUicond}
\mathbb E \left[ U_i \middle| P_{i-1} = p \right]
\leq
\frac{p}{\alpha^p C(p)} \, C_4 \, p
\end{equation}
for some constant $C_4 >0$. Therefore, if $p \leq 2 \varepsilon R^2$ and $\varepsilon$ is small enough, using the fact that $\frac{p}{\alpha^p C(p)} = \mathcal{O} (p^{1/2})$, we have
\begin{equation}
\label{eq:psmall}
\mathbb E \left[ U_i \middle| P_{i-1} = p \right]
\leq 
C_5 p^{3/2} \leq C'_5 \, R^3 \, \varepsilon
\end{equation}
where $C_5, C'_5 > 0$ are some fixed constants.

If, on the other hand $K > 1$, equations \eqref{eq:EUicond} and \eqref{eq:lawPerim} give
\begin{align*}
\mathbb E \left[ U_i \, \mathbf{1}_{\{P_{i-1} \geq K \, R^2\}} \right]
& \leq 
C_6 \sum_{p \geq K \, R^2} p^2 \left(1-\frac{1}{i^2}\right)^{p-1} \frac{1}{i^3},\\
&\leq C'_6 \, i \sum_{p \geq K \, R^2} \frac{p^2}{i^4} e^{-p/i^2}.
\end{align*}
The function $u \mapsto u^2 e^{-u}$ being decreasing for $u > 2$, the last inequality transforms to
\begin{align*}
\mathbb E \left[ U_i \, \mathbf{1}_{\{P_{i-1} \geq K \, R^2\}} \right]
&\leq C'_6 \, i \int_{u \geq K \, R^2} \frac{u^2}{i^4} e^{-u/i^2} \, du,\\
&\leq C'_6 \, i^3 \int_{u \geq \frac{K \, R^2}{i^2}} u^2 e^{-u} \, du.
\end{align*}
Since we only consider $i \in \{1, \ldots, R \}$, we have
\begin{align}
\label{eq:Plarge}
\mathbb E \left[ U_i \, \mathbf{1}_{\{P_{i-1} \geq K \, R^2\}} \right]
&\leq C'_6 \, i^3 \int_{u \geq K} u^2 e^{-u} \, du. \notag\\
& \leq C_7 \, R^3 \, K^2 e^{-K} \leq C'_7 \, R^3 \, \varepsilon
\end{align}
for $K = \varepsilon^{-1}$ and $\varepsilon$ small enough.

\bigskip

Finally, if $p = \lfloor \ell R^2 \rfloor$ for some $\ell \in [2 \varepsilon, K]$, we have using \eqref{eq:espU}:
\begin{align}
\label{eq:mediump}
\mathbb E & \left[ U_i \, \mathbf{1}_{\left\{P_i \geq p -\varepsilon R^2\right\}} \middle| P_{i-1} = p \right] \notag\\
& \quad \quad \leq C_1 p^{1/2} [u^p] \frac{\varphi^{p +1 - \lfloor \varepsilon R^2 \rfloor}}{(1 - \varphi)^{3/2}} + C_2 p^{3/2} [u^p] \frac{\varphi^{p - \lfloor \varepsilon R^2 \rfloor}}{(1 - \varphi)^{1/2}} + C_3 p^{1/2} [u^{p-2}] \frac{\varphi^{p - \lfloor \varepsilon R^2 \rfloor} \varphi''}{(1 - \varphi)^{3/2}}
\end{align}
where we also used the fact that
\[
\frac{1}{p} \alpha^p C(p) \sim_{p \to \infty} \frac{C}{p^{1/2}}.
\]
The same methods of singularity analysis than the ones used in the proof of Corollary \ref{cor:slicelim} give, as $R \to \infty$,
\begin{align*}
\left[u^{\lfloor \ell R^2 \rfloor}\right] \frac{\varphi^{\lfloor \ell R^2 \rfloor + 1 - \lfloor \varepsilon R^2 \rfloor}}{(1 - \varphi)^{3/2}}
& \sim
\frac{1}{2i \pi \ell R^2} \oint_{\mathcal{H}} e^{-z \frac{\varepsilon}{\ell}} \left( \frac{-z}{\ell R^2}\right)^{-{3/2}} \, dz = \frac{R \, \varepsilon^{1/2}}{\Gamma(3/2)},\\
\left[u^{\lfloor \ell R^2 \rfloor}\right] \frac{\varphi^{\lfloor \ell R^2 \rfloor - \lfloor \varepsilon R^2 \rfloor}}{(1 - \varphi)^{1/2}}
& \sim
\frac{1}{2i \pi \ell R^2} \oint_{\mathcal{H}} e^{-z \frac{\varepsilon}{\ell}} \left( \frac{-z}{\ell R^2}\right)^{-1/2} \, dz
= \frac{1}{\Gamma (1/2) \, R \, \varepsilon^{1/2}},\\
\left[u^{\lfloor \ell R^2 \rfloor -2}\right] \frac{\varphi^{\lfloor \ell R^2 \rfloor - \lfloor \varepsilon R^2 \rfloor} \varphi''}{(1 - \varphi)^{3/2}}
& \sim
\frac{3}{4i \pi \ell R^2} \oint_{\mathcal{H}} e^{-z \frac{\varepsilon}{\ell}} \left( \frac{-z}{\ell R^2}\right)^{-2} \, dz = \frac{3 \, R^2 \, \varepsilon}{2 \Gamma(2)}.
\end{align*}
These last three asymptotic behaviors and \eqref{eq:mediump} finally give, for any $p \in [2 \varepsilon R^2, K\, R^2]$,
\begin{align}
\label{eq:pmedium}
\mathbb E & \left[ U_i \, \mathbf{1}_{\left\{P_i \geq p -\varepsilon R^2\right\}} \middle| P_{i-1} = p \right]
\leq C'_1 \, \varepsilon^{1/2} \, K^{1/2} \, R^2 + C'_2 \, \varepsilon^{-1/2} K^{3/2} \, R^2 + C'_3 \, \varepsilon \, K^{1/2} \, R^3.
\end{align} 

\bigskip
Combining \eqref{eq:psmall}, \eqref{eq:Plarge} and \eqref{eq:pmedium} give, for every $R$,
\[
R^{-4} \mathbb E \left[ V_R^{\leq \varepsilon} \right] \leq C \, \varepsilon^{1/2},
\]
and thus, for every $\delta >0$, we have
\begin{equation}
\label{eq:vrepsilon}
\sup_{R\geq 1} \,
\mathbb P \left( \sup_{x \in [0,1]} \left|R^{-4} V_{\lfloor x R \rfloor} - R^{-4} V_{\lfloor x R \rfloor}^{\leq \varepsilon} \right| > \delta \right)
\xrightarrow[\varepsilon \to 0]{} 0.
\end{equation}

\bigskip

We now turn to $V_r^{> \varepsilon}$ and use the reasoning of the proof of Theorem 1 of \cite{CLG} (we give the full reasoning for the sake of completness). Denote by $x_1, x_2, \ldots $ the jump times of $\mathcal L$ before time $1$. For every $r \geq 1$, let $\ell_1^{(r)}, \ldots , \ell_r^{(r)}$ be the integers $i \in \{1, \ldots ,r \}$ listed in increasing order of the quantities $P_i - P_{i-1}$ (and the usual order of $\mathbb N$ for indices such that $P_i - P_{i-1}$ is equal to a given value). It follows from the convergence \eqref{eq:cvperim} that, for every integer $K \geq 1$,
\begin{align}
\label{eq:jumps}
&\left(
R^{-1} \ell_1^{(R)}, \ldots , R^{-1} \ell_K^{(R)},
R^{-2} \left( P_{\ell_1^{(R)}} - P_{\ell_1^{(R)}-1} \right), \ldots,
R^{-2} \left( P_{\ell_K^{(R)}} - P_{\ell_K^{(R)}-1} \right)
\right) \notag \\
& \quad \quad \quad
\xrightarrow[R \to \infty]{(d)}
\left(
x_1, \ldots, x_K,
3^2 \cdot \Delta \mathcal L_{x_1}, \ldots,
3^2 \cdot \Delta \mathcal L_{x_K} \right),
\end{align}
and this convergence holds jointly with the convergence \eqref{eq:cvperim}. In addition, using Corollary \ref{cor:hulldiff}, we also get
\begin{equation}
\label{eq:jumpsvolcv}
\left( 
\frac{U_{\ell_1^{(R)}}}{\left( P_{\ell_1^{(R)}} - P_{\ell_1^{(R)}-1}\right)^2}, \cdots,
\frac{U_{\ell_K^{(R)}}}{\left( P_{\ell_K^{(R)}} - P_{\ell_K^{(R)}-1}\right)^2}
\right)
\xrightarrow[R \to \infty]{(d)}
\left(4 \cdot 3^3 \cdot \xi_1, \ldots,  4 \cdot 3^3 \cdot \xi_K \right),
\end{equation}
jointly with the convergences \eqref{eq:cvperim} and \eqref{eq:jumps},
where the random variables $\xi_i$ are independent copies of the random variable $\xi$ of Corollary \ref{cor:hulldiff}, and independent of the process $\mathcal L$.

Chosing $K$ sufficiently large such that the probability of $|\Delta \mathcal L _{x_K}| <  \varepsilon / (2 \cdot 3^2)$ is close to $1$, we can combine \eqref{eq:jumps} and \eqref{eq:jumpsvolcv} to obtain the joint convergence
\[
\left(R^{-2} P_{\lfloor R x \rfloor}, R^{-4} V_{\lfloor R x \rfloor}^{> \varepsilon} \right)_{x \in [0,1]}
\xrightarrow[R \to \infty]{(d)}
\left(3^2 \cdot \mathcal L_x,4 \cdot 3^3 \cdot \mathcal M^{\varepsilon}_x\right)_{x \in [0,1]},
\]
where the process $\left( \mathcal M^{\varepsilon}_x\right)_{x \in [0,1]}$ is defined by
\[
\mathcal M^{\varepsilon}_x = 
\sum_{i \geq 1} \mathbf{1}_{\left\{x_i \leq x, \,  | \Delta \mathcal L _{x_i}| > \varepsilon / (2 \cdot 3^2) \right\}} \xi_i \left( \Delta \mathcal L_{x_i}\right)^2.
\]
It is easy to verify that, for every $\delta > 0$,
\[
\mathbb P \left( \sup_{x \in [0,1]} |\mathcal M_x - \mathcal M^{\varepsilon}_x| > \delta \right)
\xrightarrow[\varepsilon \to 0]{} 0
\]
and the final result follows from \eqref{eq:vrepsilon}.
\end{proof}

\section{Geodesic slices}
\label{sec:slices}

\subsection{Leftmost geodesics, slices and skeletons}

Fix $r > 0$ and $v \in \partial B_r^\bullet (\mathbf T_\infty)$. There are several geodesic paths from $v$ to the root vertex and we will distinguish a canonical one, called the leftmost geodesic. Informally, it is constructed from the following local rule: at each step, take the leftmost available neighbour that takes you closer to the root. More precisely, the vertex $v \in \partial B_r^\bullet (\mathbf T_\infty)$ is connected to several vertices of $\partial B_{r-1}^\bullet (\mathbf T_\infty)$ and we can enumerate them in clockwise order, starting from the first one after the edge of $\partial B_r^\bullet (\mathbf T_\infty)$ whose initial vertex is $v$. The first step of the leftmost geodesic from $v$ to the root vertex is the last edge appearing in this enumeration and the path is constructed by induction. Notice that the first step of the leftmost geodesic is an edge of the down triangle associated to the edge of $\partial B_r^\bullet (\mathbf T_\infty)$ on the left hand side of $v$ (see Figure \ref{fig:slice} for an illustration).

Now pick $v , v' \in \partial B_r^\bullet (\mathbf T_\infty)$, the two leftmost geodesics started respectively at $v$ and $v'$ will coalesce at a vertex denoted by $v \wedge v'$. The geodesic slice $\mathbf S(r,v,v')$ is the submap of $B_r^\bullet (\mathbf T_\infty)$ bounded by these two paths and the part of $\partial B_r^\bullet (\mathbf T_\infty)$ going from $v$ to $v'$ (recall that $\partial B_r^\bullet (\mathbf T_\infty)$ is oriented so that $B_r^\bullet (\mathbf T_\infty)$ lies on its right hand side). As a consequence of the definition of leftmost geodesics, the slice $\mathbf S(r,v,v')$ is completely described by the trees of the skeleton of $B_r^\bullet (\mathbf T_\infty)$ whose root lies, following the orientation of $\partial B_r^\bullet (\mathbf T_\infty)$, between $v$ and $v'$. Indeed, it is composed of the down triangles and the slots associated to the vertices of these trees. Figure \ref{fig:slice} contains an illustration of this fact.

\begin{figure}[ht!]
\begin{center}
\includegraphics[width=0.8\textwidth]{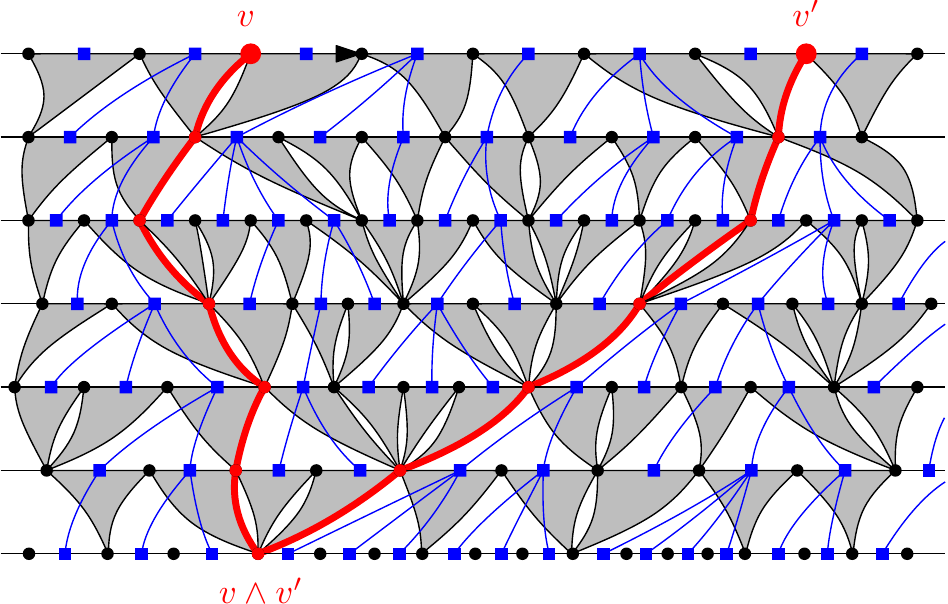}
\caption{\label{fig:slice}In red, two leftmost geodesic paths to the root started respectively at $v$ and $v'$, up to their coalescence point $v \wedge v'$. The geodesic slice $\mathbf{S}(r,v,v')$ is is the part of the map lying inside the two red paths and below the path joining $v$ and $v'$.}
\end{center}
\end{figure}

\subsection{Volume of slices}

\begin{proof}[Proof of Theorem \ref{th:slices}]
Since, for any $v,v' \in \partial B_r^\bullet (\mathbf T_\infty)$, the slice $\mathbf S(r,v,v')$ corresponds to the trees of $\mathrm{Skel} \left( B_r^\bullet (\mathbf T_\infty) \right)$ whose root lie between $v$ and $v'$, we need to identify these trees. Indeed, the first tree of $\mathrm{Skel} \left( B_r^\bullet (\mathbf T_\infty) \right)$ plays a special role (it is the only one of height $r$) and the geometry of the slice is not the same whether this tree is rooted between $v$ and $v'$ or not. Equivalently, this means that the slice containing the root vertex of $\mathbf T_{\infty}$ will play a special role.

We denote by $F=(\tau_1, \ldots , \tau_q)$ the skeleton of $B_r^\bullet (\mathbf T_\infty)$. Recall that it is an ordered forest, and more precisely a $(r,1,q)$-admissible forest. The vertex $v_1$ is the vertex on the left-hand side of the root of $\tau_i$ for some $i$ between $1$ and $q$, and the part of the skeleton describing the slice $\mathbf S(r,v_j,v_{j+1})$ is the ordered forest
\[F_{i,j}=(\tau_{i+q_1 + \cdots + q_{j-1}}, \ldots ,\tau_{i+q_1 + \cdots + q_{j}-1}),\]
where $\tau_{q+k} = \tau_k$ for every $k \in \{ 1,\ldots ,q\}$ (we also always set $v_{n+1}=v_1$). The vertex $v_1$ being chosen uniformly, this happens with probability $\frac{1}{q}$ for every $i \in \{ 1,\ldots q\}$.

Now, fix $\Delta$ a triangulation of the $(r,1,q)-$cylinder with skeleton $F = (\tau_1,\ldots,\tau_q)$. For $v,v'\in \partial \Delta$, we denote by $\Delta(v,v')$ the geodesic slice define by the arc from $v$ to $v'$ and the two leftmost geodesic started respectively ar $v$ and $v'$. If $v_1$ chosen uniformly and then $(v_2,\ldots,v_n)$ are such that the length of the arc for $v_j$ to $v_{j+1}$ along $\partial \Delta$ has length $q_j$, then, for any $s_1,\ldots ,s_n \in [0,1]$
\[
\mathbb{E} \left[ \prod_{j=1}^n s_j^{|\Delta(v_j,v_{j+1})| - d(v_j,v_j \wedge v_{j+1}) -1}\right] = \frac{1}{q} \sum_{i=1}^{q} \, \prod_{j=1}^n \prod_{v \in F_{i,j}^\star} s_j^{n_v + 1},
\]
where the expectation takes into account only the randomness of $v_1$ (the map $\Delta$ is deterministic here). This is where it is easier to consider $|\Delta(v_j,v_{j+1})| - d(v_j,v_j \wedge v_{j+1}) -1$ instead of simply $|\Delta(v_j,v_{j+1})|$. Indeed, in the previous formula, the terms $s^{n_v +1}$ count the number of inner vertices in blocs as well as the top vertex of each block. This means that every vertex of the leftmost geodesic on the right hand side of the slice is not counted explaining the deduction of $d(v_j,v_j \wedge v_{j+1}) + 1$ vertices in to size of the slice. In order to count this vertices we would have to keep track of the the height of each slice (namely $d(v_j,v_j \wedge v_{j+1})$). This is not much harder to do, but it leads to a much more complicated formula and does not have a lot of benefits.

Lemma \ref{lem:lawhull} gives
\begin{align*}
& \mathbb{E} \left[
\prod_{j=1}^n s_j^{|\mathbf S(r,v_j,v_{j+1})| - d(v_j,v_j \wedge v_{j+1}) -1}
\mathbf{1}_{\left\{ B_r^\bullet (\mathbf T_\infty) = \Delta \right\}}
\right]\\
& \quad \quad = \frac{\alpha^q C(q)}{\alpha C(1)} \frac{1}{q} \sum_{i=1}^{q} \,
\prod_{j=1}^n
\, \prod_{v \in F_{i,j}^\star} \alpha^{c(v) -1} (\rho s_j)^{n_v + 1},\\
& \quad \quad = \frac{\alpha^q C(q)}{\alpha C(1)} \frac{1}{q} \sum_{i=1}^{q} \,
\prod_{j=1}^n
\, t_j^{q_j - \mathbf{1}_{\{\tau_1 \in F_{i,j}\}}} \prod_{v \in F_{i,j}^\star} \left( \alpha t_j \right)^{c_v -1} (\rho s_j)^{n_v + 1}.
\end{align*}
Summing over every triangulation $\Delta$ having $F$ as skeleton then gives
\begin{align*}
& \mathbb{E} \left[
\prod_{j=1}^n s_j^{|\mathbf S(r,v_j,v_{j+1})| - d(v_j,v_j \wedge v_{j+1}) -1}
\mathbf{1}_{\left\{ \mathrm{Skel} \left( B_r^\bullet (\mathbf T_\infty) \right) = F \right\}}
\right] = \frac{\alpha^q C(q)}{\alpha C(1)} \frac{1}{q} \sum_{i=1}^{q} \,
\prod_{j=1}^n
\, t_j^{q_j -\mathbf{1}_{\{\tau_1 \in F_{i,j}\}}} \prod_{v \in F_{i,j}^\star}
[u^{c_v}] \varphi_{t_j} (u).
\end{align*}
Finally summing over every admissible forest yields
\begin{align*}
& \mathbb{E} \left[
\prod_{j=1}^n s_j^{|\mathbf S(r,v_j,v_{j+1})| - d(v_j,v_j \wedge v_{j+1}) -1}
\mathbf{1}_{\left\{ |\partial B_r^\bullet (\mathbf T_\infty) | = q \right\}}
\right] \\
& \quad \quad = \frac{\alpha^q C(q)}{\alpha C(1)} \frac{1}{q} 
\sum_{F \in \mathcal{F}(r,1,q)} \,
\sum_{i=1}^{q} \,
\prod_{j=1}^n
\,
t_j^{q_j -\mathbf{1}_{\{\tau_1 \in F_{i,j}\}}} \prod_{v \in F_{i,j}^\star}
[u^{c_v}] \varphi_{t_j} (u),\\
& \quad \quad = \frac{\alpha^q C(q)}{\alpha C(1)} \frac{1}{q} 
\sum_{F \in \mathcal{F}'(r,1,q)} \,
\prod_{j=1}^n
\, t_j^{q_j -\mathbf{1}_{\{h(F_{1,j})=r\}}} \prod_{v \in F_{1,j}^\star}
[u^{c_v}] \varphi_{t_j} (u),
\end{align*}
where $h(\cdot)$ denotes the maximal height of a forest. If, for $k \geq 1$, we denote by $\mathcal{F}_k^{r}$ the set of all ordered forests of $k$ trees with maximal height stricly less than $r$, we get
\begin{align*}
& \mathbb{E} \left[
\prod_{j=1}^n s_j^{|\mathbf S(r,v_j,v_{j+1})| - d(v_j,v_j \wedge v_{j+1}) -1}
\mathbf{1}_{\left\{ |\partial B_r^\bullet (\mathbf T_\infty) | = q \right\}}
\right] \\
& \quad \quad = \frac{\alpha^q C(q)}{\alpha C(1)} \frac{1}{q}
\sum_{k=1}^n
\left(
t_k^{q_k - 1} \sum_{F_{1,k} \in \mathcal{F}'(r,1,q_k)} \, \prod_{v \in F_{1,k}^\star} [u^{c_v}] \varphi_{t_k} (u)
\right) \times
\prod_{j\neq k}
\left( t_j^{q_j}
\sum_{F_{1,j} \in \mathcal{F}^r_{q_j}} \, \prod_{v \in F_{1,j}} [u^{c_v}] \varphi_{t_j} (u) \right),\\
& \quad \quad = \frac{\alpha^q C(q)}{\alpha C(1)} \frac{1}{q} \, 
\left(
\prod_{j=1}^n \left(t_j \, \varphi_{t_j}^{\{r\}}(0) \right)^{q_j}
\right) \times
\sum_{k=1}^n \frac{1}{t_k} \, 
\frac{[u]\left(\varphi_{t_k}^{\{r\}}(u) \right)^{q_k}}{\left(\varphi_{t_k}^{\{r\}}(0) \right)^{q_k}},
\\
& \quad \quad = \frac{\alpha^q C(q)}{\alpha C(1)} \frac{1}{q} \, 
\left(
\prod_{j=1}^n \left(t_j \, \varphi_{t_j}^{\{r\}}(0) \right)^{q_j}
\right) \times
\sum_{k=1}^n  \frac{q_k}{t_k} \, 
\frac{\varphi_{t_k}^{\{r\}'}(0)}{\varphi_{t_k}^{\{r\}}(0)}.
\end{align*}
Finally we have
\begin{align*}
& \mathbb{E} \left[
\prod_{j=1}^n s_j^{|\mathbf S(r,v_j,v_{j+1})| - d(v_j,v_j \wedge v_{j+1}) -1}
\middle|  |\partial B_r^\bullet (\mathbf T_\infty) | = q
\right] \\
& \qquad \qquad \qquad =
\left(
\prod_{j=1}^n \left(t_j \, \frac{\varphi_{t_j}^{\{r\}}(0)}{\varphi^{\{r\}}(0)} \right)^{q_j}
\right) \times
\sum_{k=1}^n \frac{q_k}{q} \frac{1}{t_k} \, 
\frac{\varphi_{t_k}^{\{r\}'}(0)}{\varphi^{\{r\}'}(0)}
\frac{\varphi^{\{r\}}(0)}{\varphi_{t_k}^{\{r\}}(0)}
\end{align*}
giving the result.
\end{proof}

\begin{proof}[Proof of Corollary \ref{cor:slicelim}]
This is a direct consequence of Theorem \ref{th:slices} using the same asymptotics as in the proof of Corollary \ref{cor:hullcond}.
\end{proof}

\bibliographystyle{plain}
\bibliography{biblio}

\end{document}